\documentclass[10pt]{article}

\usepackage[margin=1in]{geometry}  
\usepackage{graphicx}              
\usepackage{amsmath}
\usepackage{mathrsfs}              
\usepackage{amsfonts}              
\usepackage{amsthm}
\usepackage{tikz}               
\usetikzlibrary{arrows,positioning}
\usetikzlibrary{calc}
\usepackage{algorithm,multirow,booktabs,siunitx}
\usepackage{pdflscape}

\usepackage{algpseudocode}

\usepackage{algorithm}
\usepackage{algpseudocode}

\makeatletter
\def\BState{\State\hskip-\ALG@thistlm}
\makeatother

\usepackage{fullpage}

 \usepackage{times}
\newcommand*{\mynum}[1]{\num[output-decimal-marker={.},
   round-mode=places,
 round-precision=2,
  group-digits=false]{#1}}

\newtheorem{thm}{Theorem}[section]

\newtheorem{cor}[thm]{Corollary}

\newtheorem{remark}[thm]{Remark}

\DeclareMathOperator{\id}{id}

\newcommand{\bd}[1]{\mathbf{#1}}  
\newcommand{\RR}{\mathbb{R}}      
\newcommand{\ZZ}{\mathbb{Z}}      
\newcommand{\comb}[2]{\binom{#1^2 + #2^2}{#1+#2}}
\DeclareMathAlphabet{\mathbmit}{OML}{cmm}{b}{it}
\renewcommand{\vec}[1]{\mathbmit{#1}}
\let\matr\vec
\def\th#1{{#1}\textsuperscript{th}}

\begin{document}

\nocite{*}

\title{Representations of  quadratic combinatorial optimization problems: A case study using the quadratic set covering problem}

\author{Abraham P. Punnen and  Pooja Pandey \thanks{ Corresponding author. Email: poojap@sfu.ca}\\
Department of Mathematics, Simon Fraser University\\
 250 - 13450 – 102nd Avenue, Surrey, BC, V3T 0A3, Canada}

\maketitle

\begin{abstract}

The objective function of a quadratic combinatorial optimization problem (QCOP) can be represented by two data points, a quadratic cost matrix $\matr{Q}$ and a linear cost vector $\vec{c}$. Different, but equivalent, representations of the pair $(\matr{Q},\vec{c})$ for the same QCOP are well known in literature. Research papers often state that without loss of generality we assume $\matr{Q}$ is symmetric, or upper-triangular or  positive semidefinite, etc. These representations however have inherently different properties. Popular general purpose 0-1 QCOP solvers such as GUROBI and CPLEX do not suggest a preferred representation of $\matr{Q}$ and $\vec{c}$. Our experimental analysis discloses that GUROBI prefers the upper triangular representation of the matrix $\matr{Q}$ while CPLEX prefers the symmetric representation in a statistically significant manner. Equivalent representations, although preserve optimality, they could alter the corresponding lower bound values obtained by various lower bounding schemes. For the natural lower bound of a QCOP, symmetric representation produced tighter bounds, in general. Effect of equivalent representations when CPLEX and GUROBI run in a heuristic mode are also explored. Further, we review various equivalent representations of a QCOP from the literature that have theoretical basis to be viewed as 'strong'  and provide new theoretical insights for generating such equivalent representations making use of constant value property and diagonalization (linearization) of QCOP instances.
\end{abstract}

\smallskip
\noindent \textbf{Keywords.}  0-1 quadratic programming, experimental analysis of algorithms, set covering problem, equivalent formulations, combinatorial optimization.

\section{Introduction}
Let $E = \{1,2,\dotsc,n\}$ be a finite set and $ \hat{F} $ be a family of subsets of $E$. For each $j \in E$, a cost $c_j$ is prescribed. Further, for each $(i,j) \in E \times E$,  a cost $q_{ij}$ is also prescribed. Note that any $S \in \hat{F}$ can be represented by its $0-1$ incidence vector  $ \vec{x} = (x_1,\hdots,x_n)$ where $x_j=1$ if and only if $j \in S$. Thus $\hat{F}$ can be represented as $\mathscr{F}= \{ \vec{x} \in \{0,1\}^n: \vec{x}  \text{  is an incidence vector of some  } S \in \hat{F} \}$. Let $\matr{Q}$ be the $n \times n $ matrix such that its $(i,j)^{th}$ element is $q_{ij}$ and $\vec{c}=(c_1,\hdots,c_n)$. Then, the \textit{quadratic combinatorial optimization problem} (QCOP) is to
\begin{align*}
 \mbox{Minimize~ } & \vec{c}\vec{x} + \vec{x}^T\matr{Q}\vec{x}\\
 \mbox{Subject to } & \vec{x} \in \mathscr{F}
\end{align*}
or equivalently
\begin{align*}
 \mbox{Minimize~ } & \sum_{i \in S} c_i+  \sum_{i \in S}   \sum_{j \in S} q_{ij}\\
 \mbox{Subject to } & S \in \hat{F}.
\end{align*}

The well known quadratic assignment problem~\cite{1998Cela,1963Lwaler}, the quadratic unconstrained binary optimization problems (QUBO) \cite{2014Gary}, and the quadratic knapsack problem~\cite{2007Pisinger} are special cases of QCOP. Other examples of QCOP  include the  quadratic travelling salesman problem \cite{2014Fischer,2017Oswin,2017Punnen}, the quadratic shortest path problem~\cite{2017Hu,2015Rostami}, the general quadratic $0-1$ programming problem \cite{2004Billionnet,be1,ba1,ba10,gal,ba5,ba2,2007Sherali}, the quadratic spanning tree problem \cite{1992Assad,{2017Ante}}, quadratic set covering problem~\cite{bruno,2005hammer,2017Pandey}, 0-1 bilinear programs \cite{ac1,gu1}, and combinatorial optimization problems with interaction costs~\cite{2017Lendl}.\\

When the elements of $\mathscr{F}$ are represented by a collection of linear  constraints in binary variables, QCOP can be solved using general purpose binary quadratic programming solvers such as CPLEX \cite{CPLEX12.5} or GUROBI \cite{GUROBI605}. The matrix $\matr{Q}$ associated with a QCOP can be represented in many different but equivalent forms using appropriate  transformations on $\matr{Q}$ and $\vec{c}$. For example, it is possible to force $\matr{Q}$  to have  properties such as $\matr{Q}$ is positive semidefinite~\cite{1970Hammer},  negative semidefinite~\cite{1970Hammer}, symmetric with diagonal entries zero~\cite{1970Hammer},  upper (lower) triangular~\cite{gal} etc.,  so that the resulting problem is equivalent to the given QCOP. Many authors use one of these equivalent representations to define a QCOP.  This raises the question: ``Which representation of $\matr{Q}$ is better from a computational point of view?'' The answer to this question depends on how one defines a `better representation'. Extending the work of Hammer and Rubin~\cite{1970Hammer}, Billionnet et al~\cite{be1} used diagonal perturbations in an 'optimal' way to create strong reformulations of QUBO.  Billionnet~\cite{2004Billionnet}, Billionnet et al~\cite{ba1,ba10}, and P\"{o}rn et al.~\cite{ba2} extended this further to include perturbations involving non-diagonal elements by making use of linear equality constraints, if any, associated with a QCOP. These reformulations  force $\matr{Q}$ to be symmetric and positive semidefinite yielding strong continuous relaxation. Galli and Letchford~\cite{gal} obtained strong reformulations using quadratic constraints of equality type. Although all these representations are very interesting in terms of obtaining strong lower bounds at the root node of a branch-and-bound search tree, they require additional computational effort that are not readily available within general purpose solvers such as CPLEX or GUROBI. To the best of our knowledge, neither CPLEX nor GUROBI makes a recommendation regarding a simple and specific representation of the $\matr{Q}$ matrix that is normally more effective for their respective solver. \\

It is not difficult to construct examples where one representation works well for CPLEX while the same representation do not work well for GUROBI and vice versa. For example, GUROBI solved a quadratic set covering instance involving 511 constraints and 210 variables in 4933 milliseconds on a PC with windows 7 operating system, Intel 4790 i7 3.60 GHz processor and 32 GB RAM. The same problem, when represented in an equivalent form with symmetry forced, GROBI could not solve in 3 hours. CPLEX solved the problem in 23674 milliseconds and for an equivalent representation with symmetry forced, it solved in 21588 milliseconds.
For another class of problems, GUROBI solved random non-diagonal reformulations efficiently, while structured equivalent formulations where $\matr{Q}$ having properties such as symmetry, triangularity, positive semidefiniteness or  negative semidefiniteness, could not solve many of the problems in this class (see Table 4 and Table 5). CPLEX however solved all these reformulations, although the time taken was larger than that of GUROBI for random perturbations. \\

We also could not find anything in the literature regarding a preferred representation of $\matr{Q}$ for solving QCOP established through systematic experimental analysis. Motivated by this, we investigate on the representation of the $\matr{Q} $ matrix for a QCOP. Unlike the way  interesting theoretical works reported in ~\cite{2004Billionnet,be1,ba1,ba10,gal,ba2},  we are not attempting to develop optimal representation based on some desirability criteria. Our experimental results in Table 4 and Table 5 substantiate the merit of investigating this line of reasoning as well.  Consequently, we present various transformations that provide equivalent representations of the problem, not necessarily 'optimal' ones.  From these representations, we identify six simple and well known classes that are compared using CPLEX and GUROBI. The experimental study discloses that CPLEX prefers symmetric or symmetric with a diagonal perturbation yields a positive definite matrix, whereas GUROBI prefers an upper triangular $\matr{Q}$ matrix. Although there are outliers, statistical significance of these observations are established through Wilcoxin test \cite{1945Wilcoxon}.  We also propose ways to construct strong reformulations making use of constant value property~\cite{cb2,cb1} associated with linear combinatorial optimization problems and the concept of diagonalizable (linearizable) cost matrices associated with a QCOP~\cite{2017Ante,2017Hu,2011Kabadi,2017Punnen}. \\

Equivalent representation of the data could also influence lower bound calculations for a QCOP. To demonstrate its  impact, we used a generalization of the well know Gilmore-Lawler lower bound~\cite{1962Gilmore,1963Lwaler} and its variations~\cite{1992Assad}. Our experiments show that for most of the test problems we used, the strongest lower bound was obtained when used an equivalent representation where $\matr{Q}$ is  forced to be symmetric, except for one class of test problems for  which the upper triangular structure produced tighter bound.\\

To conduct experimental analysis, we selected the {\it quadratic set covering problem} (QSCP). The QSCP model have applications in the wireless local area planning  and the problem of locating access points  to guarantee full coverage~\cite{2004amaldi}.   Other application areas of QSCP include logical analysis of data \cite{2005hammer}, medicine \cite{2000boros,bruno}, facility layout problems \cite{1975bazaraa}, line planning in public transports  \cite{1990ceder, 1995israeli}  etc. Another motivation for selecting the QSCP as our test case is  that   relatively fewer computational studies  available for this model. Thus, this work also contributes to experimental analysis of exact and heuristic algorithms for the QSCP.\\

The paper is organized as follows. In section 2, we  discuss various equivalent representations of  the QCOP. Some of these representations are generated using diagonalizable (linearizable) quadratic cost matrices and linear cost vectors satisfying  constant value property. Characterization of diagonalizable cost matrices for the general QCOP and for a restricted version where feasible solutions have the same cardinality are also given. We also present a natural lower bound  for QCOP, that is valid under the equivalent transformations.  Section 3 discusses details of the experimental platform, generation of test data, and experimental results  on QSCP  using CPLEX12.5  and GUROBI6.0.5 comparing selected equivalent representations for exact and heuristic solutions. Experimental analysis using the  natural lower bound is also given in this section followed  by concluding remarks  in Section 4.\\

Throughout the paper, we use the following notations. For a given $\mathscr{F}$, a QCOP can be represented by $(\matr{Q}, \matr{c})$. The matrix $\matr{Q}$ is called the \textit{quadratic cost matrix} and the vector $\vec{c}$ is called the \textit{linear cost vector}. For an instance $(\matr{Q}, \vec{c})$ of a QCOP and an $\vec{x}\in \mathscr{F}$, $f(\matr{Q}, \vec{c}, \vec{x}) = \vec{x}^T\matr{Q}\vec{x} + \vec{c}\vec{x}$.  $\mathscr{F}_R$ is the continuous relaxation of $\mathscr{F}$. i.e. $\mathscr{F}_R$ is obtained by replacing the constraints $x_j\in \{0,1\}$ in the definition of $\mathscr{F}$ by $0\leq x_j \leq 1$. For $\vec{x} \in \mathscr{F}_R$ we denote $f_R(\matr{Q}, \vec{c}, \vec{x}) = \vec{x}^T\matr{Q}\vec{x} + \vec{c}\vec{x}$. For any matrix $\matr{Q}$, $Diag(\matr{Q})$ is the diagonal matrix of same size as $\matr{Q}$ with its \th{$(i,i)$} element is $q_{ii}$ and $Diag(\matr{Q})$ represents the vector $ (q_{11},q_{22},\ldots ,q_{nn})$ . For any vector $\vec{a} = (a_1,a_2,\ldots ,a_n)$, $Diag(\vec{a})$ is the $n\times n$ diagonal matrix with its \th{($i,i$)} element is $a_i, i=1,2,\ldots ,n$.
 All matrices are represented using capital letters and all elements of the matrix are represented by corresponding double subscripted lower-case letters where the subscripts denoting row and column indices.  Vectors in $\mathbb{R}^n$ are represented by bold lower-case letters.  The $i$th component of vector $\vec{a}$ is $a_i$, of vector $\bar{\vec{b}}$, is $\bar{b}_i$ etc. The transpose of a matrix $\matr{Q}$ is represented by $\matr{Q}^T$.  The vector space of all real valued $n\times n$ matrices  with standard matrix addition and scalar multiplication is denoted by $\mathbb{M}^n$.

\section{Equivalent representations}

Let $(\matr{Q},\vec{c})$ be an instance of a QCOP. Then, $(\matr{Q}^1,\vec{c}^1)$ is an \textit{equivalent representation} of $(\matr{Q},\vec{c})$ if $f(\matr{Q},\vec{c},x) = f(\matr{Q}^1,\vec{c}^1,x)$ for all $x\in \mathscr{F}$. The following remark is well known.

\begin{remark}\label{rma-1}
$(\matr{Q}^T,\vec{c})$ is an equivalent representation of $(\matr{Q},\vec{c})$.
\end{remark}

\begin{thm}\label{thma-1}
If $(\matr{Q}^1,\vec{c}^1)$, $(\matr{Q}^2,\vec{c}^2), \ldots , (\matr{Q}^k,\vec{c}^k)$ are equivalent representations of an instance $(\matr{Q},\vec{c})$ of a QCOP  then $\left(\frac{1}{\sum_{i=1}^k\alpha_i}\left[\sum_{i=1}^k\alpha_i\matr{Q}^i\right],
\frac{1}{\sum_{i=1}^k\alpha_i}\left[\sum_{i=1}^k\alpha_i\vec{c}^i\right]\right)$
is also an equivalent representation of $(\matr{Q},\vec{c})$ whenever $\sum_{i=1}^k\alpha_i \neq 0$.
\end{thm}
\begin{proof}
Let $\matr{A} = \frac{1}{\sum_{i=1}^k\alpha_i}\left[\sum_{i=1}^k\alpha_i\matr{Q}^i\right]$ and $\vec{b} = \frac{1}{\sum_{i=1}^k\alpha_i}\left[\sum_{i=1}^k\alpha_i\vec{c}^i\right]$. Then
\begin{align*}
f(\matr{A},\vec{b},x) &= \vec{x}^T\matr{A}\vec{x}+\vec{b}\vec{x}
=\frac{1}{\sum_{i=1}^k\alpha_i}\sum_{i=1}^k\left(\vec{x}^T\alpha_i\matr{Q}^i\vec{x}+\alpha_i\vec{c}^i\vec{x}\right)
=\frac{1}{\sum_{i=1}^k\alpha_i}\sum_{i=1}^k\left(\alpha_if(\matr{Q}^i,\vec{c}^i,\vec{x})\right)\\
&=\frac{1}{\sum_{i=1}^k\alpha_i}\sum_{i=1}^k\left(\alpha_if(\matr{Q},\vec{c},\vec{x})\right)
=f(\matr{Q},\vec{c},\vec{x}).
\end{align*}
\end{proof}
From Remark~\ref{rma-1} and Theorem~\ref{thma-1} we have the following well-known corollary.
\begin{cor}\label{cora-1}
$\left(\frac{1}{2}\left[\matr{Q}+\matr{Q}^T\right],\vec{c}\right)$ is an equivalent representation of $(\matr{Q},\vec{c})$.
\end{cor}

We call the equivalent representation $\left(\frac{1}{2}\left[\matr{Q}+\matr{Q}^T\right],\vec{c}\right)$ given in Corollary~\ref{cora-1} the \textit{symmetrization}. This representation is well known and used extensively in literature. Since $\frac{1}{2}\left[\matr{Q}+\matr{Q}^T\right]$ is a symmetric matrix, it is sometimes viewed as a desirable representation. However, symmetrization could also result in a matrix with increased or decreased rank. Thus the equivalent representation obtained by symmetrization could have properties different  from those of  the original representation and this could impact the computational performance of different algorithms. Note that $f_R(\matr{Q},\vec{c},\vec{x}) = f_R(\matr{Q^T},\vec{c},\vec{x})=f_R(\frac{1}{2}\left[\matr{Q+Q^T}\right],\vec{c},\vec{x})$ for all $x\in \mathscr{F}_R$. Thus, the symmetrization operation also preserves the objective function value of the continuous relation of a QCOP. This property no longer holds for some other equivalent representations discussed later.\\

If one or more elements of $\matr{Q}$, say $q_{ij}$ is perturbed by $\epsilon_{ij}$ and adjusting this by subtracting $\epsilon_{ij}$ from $q_{ji}$, we immediately get an equivalent representation of $\matr{Q}$. Equivalent representations obtained this way have the  structure of $\matr{Q}^{\prime}$ discussed in the theorem below.

 \begin{thm} \label{thm-2.4}
If $\matr{Y}$ is a skew-symmetric matrix, $\matr{D}$ is a diagonal matrix, $\matr{Q}^{\prime}= \matr{Q}+\matr{Y}+\matr{D}$, and $\vec{c}^{\prime}=\vec{c}-diag(\matr{D})$, then $(\matr{Q}^{\prime},\vec{c}^{\prime})$ is an equivalent representation of $(\matr{Q},\vec{c})$.
 \end{thm}
\begin{proof}
Since $\vec{x}^T\matr{Y}\vec{x} = 0$ and $x_i^2=x_i$ for all $i=1,2,\ldots ,n$ it follows that $f(\matr{Q},\vec{c}, \vec{x}) = f(\matr{Q}^{\prime},\vec{c}^{\prime}, \vec{x})$.
\end{proof}

In the above theorem, if $\matr{Q}$ is symmetric and if we want $\matr{Q}^{\prime}$ also to be symmetric, then $\matr{Y}$ must be the zero matrix. In this case, we can choose $\matr{D} = \lambda \matr{I}$ for sufficiently large $\lambda$ to make $\matr{Q}^\prime$ a symmetric positive semidefinite matrix and hence $f_R(\matr{Q}^\prime,\vec{c}^{\prime},\vec{x})$ becomes convex.
Hammer and Ruben~\cite{1970Hammer} suggested using $\lambda$ as the negative of the smallest eigenvalue of $\matr{Q}$. Billionnet et al~\cite{be1} proposed an 'optimal' choice of the matrix $\matr{D}$ in the case of quadratic unconstrained binary optimization (QUBO) problems. Their selection of $\matr{D}$ is 'optimal' in the sense that the resulting optimal objective function value $f_R(\matr{Q}^{\prime},\vec{c}^{\prime},\vec{x})$ of the continuous relaxation is as large as possible yielding tight lower bounds.  This method extends to QCOP with appropriate restriction on the representation of $\mathscr{F}$.\\

A quadratic cost matrix $\matr{Q}$ associated with a QCOP is said to be \textit{diagonalizable} with respect to $\mathscr{F}$ if there exists a diagonal matrix $\matr{D}$ such that $\vec{x}^{T} \matr{Q} \vec{x} = \vec{x}^{T} \matr{D} \vec{x}$ for all $\vec{x} \in  \mathscr{F}$. The matrix $\matr{D}$ is called a \textit{diagonalization} of $\matr{Q}$ with respect to $\mathscr{F}$.  Here after the terminology `diagonalizable" (``diagonalization") means diagonalizable (diagonalization) with respect to  the underlying families $\mathscr{F}$.
Recall that $x_i^2=x_i$ for all $\vec{x} \in \mathscr{F}$  and hence $\vec{x}^{T} \matr{D} \vec{x} = diag(\matr{D})\vec{x}$, where $diag(\matr{D})$ is a vector of size $n$ with its $i^{th}$ element as the  $i^{th}$ diagonal entry of $\matr{D}$. Diagonalizable  matrices form a subspace of the vector space $\mathbb{M}^{n}$ of all $n \times n$ real valued matrices. The concept of diagonalization indicated here is closely related to the  linearization of some quadratic combinatorial optimization problems discussed in~\cite{2017Ante,2017Hu,2011Kabadi,2017Punnen} and for the case of binary variables, these two notions are the same.  Since the terminology  ``linearization" is also used in another context in the case of QCOP~\cite{ba3,2007Sherali}, we preferred to use the more natural and intuitive terminology diagonalization.  Note that if $\matr{Q}$ is diagonalizable then  $\matr{Q}^{T}$ and $\frac{1}{2}\left(\matr{Q}+\matr{Q}^{T}\right)$ are also diagonalizable. Also, any skew-symmetric matrix is diagonalizable and a zero matrix of the same dimension is  diagonalization of a skew symmetric matrix.

\begin{thm}\label{thma-2}  $(\matr{Q} + \matr{A}, \vec{c}-diag(D))$ is an equivalent representation of the QCOP instance $(\matr{Q}, \vec{c})$, where $\matr{A}$ is any diagonalizable  matrix associated with the QCOP and $\matr{D}$ is a diagonalization of $\matr{A}$.
\end{thm}
\begin{proof}
Since $\matr{A}$ is diagonalizable, $\vec{x}^T\matr{A}\vec{x} = diag(\matr{D})\vec{x}$ for all $\vec{x} \in \mathscr{F}$. Thus,
\begin{align*}
f(\matr{Q} + \matr{A}, \vec{c}-diag(D),\vec{x}) &= \vec{x}^T(\matr{Q} + \matr{A})\vec{x} + (\vec{c}-diag(D))\vec{x}\\
&= \vec{x}^T\matr{Q}\vec{x} + \vec{x}^T\matr{A}\vec{x} + \vec{c}\vec{x}-diag(D)\vec{x}\\
&=\vec{x}^T\matr{Q}\vec{x} + diag(D)\vec{x} + \vec{c}\vec{x}-diag(D)\vec{x}\\
&= \vec{x}^T\matr{Q}\vec{x} +  \vec{c}\vec{x} = f(\matr{Q},\vec{c},\vec{x}).
\end{align*}

\end{proof}

\begin{cor}\label{cora-2} If $\matr{A}^1,\matr{A}^2,\ldots ,\matr{A}^m$ are diagonalizable matrices associated with a QCOP and $\alpha_1,\alpha_2,\ldots ,\alpha_m$ are scalars, then  $(\matr{Q} + \sum_{i=1}^m\alpha_i\matr{A}^i, \vec{c}-\sum_{i=1}^m\alpha_i diag(D^i))$ is an equivalent representation of the QCOP instance $(\matr{Q}, \vec{c})$ where $\matr{D}^i$ is a diagonalization of $\matr{A}^i$ for $i=1,2,\ldots ,m$.
\end{cor}

We can strengthen the equivalent representation given in Corollary~\ref{cora-2} using a result by Galli and Letchford~\cite{gal}.  Since $\matr{A}^1,\matr{A}^2,\ldots ,\matr{A}^m$ are diagonalizable with respective diagonalizations $\matr{D}^1,\matr{D}^2,\ldots ,\matr{D}^m$, our QCOP satisfies the constraints
\begin{align}\label{eqp-1}
\vec{x}^T\matr{A}^i\vec{x} - Diag(\matr{D}^i)\vec{x} = 0 \mbox{ for } i=1,2,\ldots ,m,
\end{align}
 Since $\matr{A}^i$ is diagonalizable with diagonalization $\matr{D}^i$, $\frac{1}{2}\left(\matr{A}^i+(\matr{A}^i)^T\right)$ is a symmetric diagonalizable matrix with $\matr{D}^i$ as its diagonalization. Thus we can assume that $\matr{A}^i$ in equation \eqref{eqp-1} is symmetric for all $i$. Thus, we can apply the quadratic convex reformulation (QCR) technique discussed in~\cite{gal} to yield a strong reformulation making the resulting equivalent formulation have a continuous relaxation which is convex. Note that  symmetric diagonalizable matrices form a subspace of the vector space $\mathbb{M}^n$. We can use $\matr{A}^1,\matr{A}^2,\ldots ,\matr{A}^m$ discussed above as a basis of this subspace and applying QCR reformulation ~\cite{gal} to yield stronger equivalent representations. A recent related work is by Hu and Sotirov \cite{2018Hu}, that used diagonability to obtain strong lower bound for the quadratic shortest path problem on acyclic digraphs.\\

To generate equivalent representations of a  QCOP using Theorem~\ref{thma-2}, Corollary~\ref{cora-2} or  by the QCR method~\cite{gal} as discussed above, we need to identify associated diagonalizable matrices. Characterization  of diagonalizable quadratic cost matrices associated with a QCOP has been studied  by different authors for specific cases, exploiting the underlying structure of $\mathscr{F}$. This include  quadratic assignment problems \cite{2011Kabadi}, special quadratic shortest path problems \cite{2017Hu}, the quadratic spanning tree problem \cite{2017Ante}, and the quadratic traveling salesman problem \cite{2017Punnen}. However, for the general QCOP without restricting the structure of $\mathscr{F}$, the characterization of diagonalizable quadratic cost matrices do not yield rich classes like what was indicated for the special problems mentioned above. This is because QUBO is a special case of QCOP where any subset  of $E$ is feasible.

\begin{thm} \label{thma-3} A quadratic cost matrix $\matr{Q}$ associated with a QCOP is diagonalizable if and only if $\matr{Q} = \matr{Y} + \matr{U}$  where $\matr{Y}$ is a  skew-symmetric matrix and $\matr{U}$ is a diagonal matrix.
\end{thm}

\begin{proof} Since $\matr{Y}$ is skew-symmetric, $\vec{x}^{T} \matr{Y} \vec{x} =0$ for any $\vec{x} \in  \mathscr{F}$ and hence  $\matr{Q} = \matr{Y} + \matr{U}$ is diagonalizable. Further, the diagonalization of such a $\matr{Q}$ is $diag(U)$. To prove the converse, it is enough to show that for the quadratic unconstrained binary optimization problem (QUBO), if $\matr{Q}$ is diagonalizable, then $\matr{Q}$ must be of the form $ \matr{Y} + \matr{U}$. Note that the family of feasible solutions for QUBO is $\{0,1\}^n$. First, we prove that for QUBO, if a quadratic cost matrix $\matr{Q}^{'}$ is  symmetric with diagonal entries  zero is diagonalizable, then $\matr{Q}^{'}$ must be the zero matrix. Suppose that is not true. Let the $(i,j)^{th}$ element $q_{ij}^{'} \not= 0$. Then by symmetry  $q_{ji}^{'} = q_{ij}^{'} \not= 0$. Now consider the solution

\begin{equation*}
x_k =
\begin{cases}
1 &\text{for $k=\ell$ for some $\ell$}\\
0 &\text{otherwise.}
\end{cases}
\end{equation*}

\noindent Let $\matr{D}$ be a diagonalization of $\matr{Q}^{'}$.  Then $\vec{x}^{T} \matr{Q}^{'} \vec{x} =  \vec{x}^{T} \matr{D} \vec{x}$ which implies $d_{\ell\ell} =0$. Since $\ell$ is arbitrary, $\matr{D}$  must be a zero matrix. Now consider the solution

\begin{equation*}
x_k =
\begin{cases}
1 &\text{for  $k=i,j$}\\
0 &\text{otherwise.}
\end{cases}
\end{equation*}

\noindent Then  $\vec{x}^{T} \matr{Q}^{'} \vec{x} = 2 q_{ij}^{'} \not=0$. Since $\matr{Q}^{'}$ is a diagonalizable, $\vec{x}^{T} \matr{Q}^{'} \vec{x} =  \vec{x}^{T} \matr{D} \vec{x} =0$  which implies $q^{\prime}_{ij}=0$, a contradiction. Thus for any symmetric cost matrix $\matr{Q}^{'}$ with diagonal  entries zero of a QCOP, if $\matr{Q}^{'}$  is diagonalizable then $\matr{Q}^{'}$ must be zero. Now take any cost matrix $\matr{Q}$ of a QCOP that is diagonalizable.  Let $\matr{\bar{Q}} = \matr{Q} - diag(\matr{Q}).$ Then $\matr{\bar{Q}}$ is diagonalizable and hence $\matr{\hat{Q}} = \frac{1}{2}\left(\matr{\bar{Q}}+\matr{\bar{Q}}^{T}\right)$ is diagonalizable.  But $\matr{\hat{Q}}$  is symmetric with diagonal entries zero. Then $\matr{\hat{Q}}$ must be a zero matrix and hence $\matr{\bar{Q}} = - \matr{\bar{Q}}^{T}$.   Thus $\matr{\bar{Q}}$ is skew symmetric. But $\matr{Q} = \matr{\bar{Q}} + diag(\matr{Q})$ and the result follows.
\end{proof}

Note that Theorem  \ref{thm-2.4} follows on a corollary of  Theorems~\ref{thma-2} and \ref{thma-3}.


As observed earlier, imposing additional restrictions on the family of feasible solutions, more interesting characterizations for diagonalizability can be obtained~\cite{2017Ante,2017Hu,2011Kabadi,2017Punnen}. Let us now add a simple restriction that all elements of the underlying $\hat{F}$ have the same cardinality. The resulting QCOP is called the \textit{cardinality constrained quadratic combinatorial optimization problem} (QCOP-CC).\\

A matrix $\matr{P}$ is said to be a \textit{weak-sum matrix}~\cite{cb2} if there exists vectors $\vec{a},\vec{b}\in \mathbb{R}^n$ such that $p_{ij} = a_i+b_j$ for $i,j =1,2,\ldots , n, i\neq j$. Here $\vec{a}$ and $\vec{b}$ are called the \textit{generator vectors} of $\matr{P}$. Note that the sum of a weak-sum matrix and a diagonal matrix is a weak-sum matrix. For QCOP-CC we have the following characterization for diagonalizability.

\begin{thm} \label{thma-5} A quadratic cost matrix $\matr{Q}$ associated with a QCOP-CC is diagonalizable if and only if $\matr{Q} = \matr{P}+\matr{Y}$ where $\matr{P}$ is a weak-sum matrix and $\matr{Y}$ is a  skew-symmetric matrix.
\end{thm}
\begin{proof}
Let $K$ be the cardinality of elements in the underlying $\hat{F}$ defining the QCOP-CC instances. Suppose $\matr{Q} = \matr{P}+\matr{Y}$ where $\matr{P}$ is a weak-sum matrix and $\matr{Y}$ is a skew-symmetric matrix. Then
\begin{align*}
\vec{x}^T\matr{Q}\vec{x} &= \vec{x}^T\matr{P}\vec{x} + \vec{x}^T\matr{Y}\vec{x} \\
&= \sum_{i=1}^n\sum_{j=1}^n\left(a_i+b_j\right)x_ix_j -\sum_{i=1}^n\left(a_i+b_i\right)x_i + \sum_{i=1}^np_{ii}x_i\\
&= K\vec{a}\vec{x} + K\vec{b}\vec{x} - \vec{a}\vec{x} - \vec{b}\vec{x} + diag(\matr{P})\vec{x}\\
&=\left[(K-1)(\vec{a}+\vec{b})+diag(\matr{P})\right]\vec{x} = \vec{x}^T\matr{D}\vec{x},
\end{align*}
where $D$ is a diagonal matrix with $d_{ii} = (K-1)(a_i+b_i)+p_{ii}, i=1,2,\ldots ,n$. Thus $\matr{Q}$ is diagonalizable.\\

Conversely, suppose $\matr{Q}$ is diagonalizable. We will show that $\matr{Q}$ is of the required form given in the theorem. To establish this necessary condition, it is enough to establish it for a special case of QCOP(K). So, consider the quadratic minimum spanning tree problem (QMST) on a complete graph. Custic and Punnen~\cite{2017Ante} showed that a symmetric quadratic cost matrix associated with a QMST is diagonalizable if and only if it is a weak-sum matrix. Consider a quadratic cost matrix $\matr{Q}$ for the QMST. Now, $\matr{Q}$ is diagonalizable if and only if $\frac{1}{2}\left(\matr{Q}+\matr{Q}^T\right)$ is diagonalizable. Since $\frac{1}{2}\left(\matr{Q}+\matr{Q}^T\right)$ is symmetric, it follows from~\cite{2017Ante} that $\frac{1}{2}\left(\matr{Q}+\matr{Q}^T\right)$ is a weak-sum matrix. But $\matr{Q} = \frac{1}{2}\left(\matr{Q}+\matr{Q}^T\right) + \frac{1}{2}\left(\matr{Q}-\matr{Q}^T\right)$. Since $\frac{1}{2}\left(\matr{Q}-\matr{Q}^T\right)$ is skew-symmetric, the result follows.
\end{proof}

\begin{cor} \label{thma-6}Let $\matr{Y}$ be a skew-symmetric matrix, $\matr{U}$ be a diagonal matrix, and $\matr{P}$ be a weak-sum matrix with generator vectors $\vec{a}$ and $\vec{b}$. If $\matr{Q}^{\prime}= \matr{Q}+\matr{Y}+\matr{U}+\matr{P}$, and $\vec{c}^{\prime}=\vec{c}-diag(\matr{U})-(K-1)(\vec{a}+\vec{b})-diag(\matr{P})$ then
$(\matr{Q}^{\prime},\vec{c}^{\prime})$ is an equivalent representation of $(\matr{Q},\vec{c})$ for QCOP-CC, where $K$ is the fixed cardinality of elements of $\hat{F}$ defining the QCOP-CC instances.
\end{cor}

A cost vector $\vec{a}$ of a \textit{linear combinatorial optimization problem} (LCOP) satisfies \textit{constant value property} (CVP)~\cite{cb2} if $\vec{a}\vec{x}=b$ for all $\vec{x} \in \mathscr{F}$ and some constant $b$. Characterization cost matrices associated with a linear combinatorial optimization problem satisfying CVP has been studied extensively in literature for various special cases. This include the travelling salesman problem~\cite{bp1,berenguer79,gb,kp1}, assignment problem~\cite{gb}, spanning tree problem~\cite{cb2}, shortest path problem~\cite{cb2}, multidimensional assignment problem~\cite{cb2,cb1}  etc.
Consider a QCOP with family of feasible solutions $\mathscr{F}$. Suppose each of the vectors $\vec{a}^1,\vec{a}^2,\ldots ,\vec{a}^p$ satisfies CVP with respect to $\mathscr{F}$. Then  $\vec{a}^i\vec{x} =b_i$ for all $\vec{x}\in \mathscr{F}$, $i=1,2,\ldots ,p$ where $b_i, i=1,2,\ldots p$ are some constants. Using these natural equalities, we can apply the QCR method of Billionnet et al~\cite{2004Billionnet,ba10} or P\"{o}rn et al~\cite{ba2} to generate strong reformulations of QCOP with appropriate restrictions on $\mathscr{F}$. It may be noted that the vectors satisfying CVP for an LCOP form a subspace of $R^n$. Thus, in particular, if we choose $\vec{a}^1,\vec{a}^2,\ldots ,\vec{a}^p$ as a basis for this subspace, strong reformulations can be obtained using the QCR method of Billionnet et al~\cite{ba10} or of P\"{o}rn et al~\cite{ba2}.\\

Vectors satisfying CVP can also be used to generate diangonalizable matrices in a natural way which in turn can be used to generate strong reformulations as discussed earlier. To see this, let $\vec{a}^1,\vec{a}^2,\ldots ,\vec{a}^n$ be a collection of cost vectors (not necessarily distinct) satisfying CVP with respect to $\mathscr{F}$ with respective constant values $b_1,b_2,\ldots ,b_n$. Let $\matr{A}$ be the matrix with its \th{$i$} row is $\vec{a}^i$. Then $\matr{A}$ is diagonalizable and $\matr{D}$ with $d_{ii}=b_i$ is its diagonalization. Diagonalizable matrices generated this way can be used to obtain equivalent representations as discussed in Theorem~\ref{thma-2}. Let us now observe that vectors satisfying CVP can also be used to obtain Billionnet et al~\cite{ba10} type equivalent representations for QCOP.\\

Suppose $\vec{a}=(a_1,a_2,\ldots ,a_n)\in \mathbb{R}^n$ satisfies CVP for the family $\mathscr{F}$ of feasible solutions of a QCOP with $b$ as the constant value. Then $\vec{a}\vec{x} = b$ for all $\vec{x} \in \mathscr{F}$. Create $n$ copies of this equation and multiply both sides of the \th{$i$} equation by  $\alpha_ix_i$, for $i=1,2,\ldots ,n$ where $\alpha_i$ is a scalar. Adding these equations give $\sum_{i=1}^n\sum_{j=1}^n\alpha_ia_jx_ix_j = b\sum_{i=1}^n\alpha_ix_i$. This can be written as $\vec{x}^T\left(\vec{\alpha}^T\vec{a}\right)\vec{x} = b\vec{\alpha}\vec{x}$, where $\vec{\alpha}=(\alpha_1,\alpha_2,\ldots ,\alpha_n)$.

\begin{thm}
Let $\matr{G} = \frac{1}{2}\left(\vec{\alpha}^T\vec{a} + \vec{a}^T\vec{\alpha}\right)$. Then $(\matr{Q}+\matr{G}+\matr{D}, \vec{c}-b\vec{\alpha}-diag(\matr{D}))$ is an equivalent representation of $(\matr{Q},\vec{c})$ where $\matr{D}$ is any diagonal matrix.
\end{thm}

The proof of the theorem follows from the previous discussions. As a corollary, we have

\begin{cor}\label{cpa}
Let $\vec{a}^1,\vec{a}^2,\ldots ,\vec{a}^p$ be $p$ vectors satisfying CVP for solutions in $\mathscr{F}$ with respective constant values $b_1,b_2,\ldots ,b_n$ and $\vec{\alpha}^1,\vec{\alpha}^2,\ldots ,\vec{\alpha}^p$ are vectors in $\mathbb{R}^n$. Let $\matr{G}^i = \frac{1}{2}\left((\vec{\alpha}^i)^T\vec{a}^i + (\vec{a}^i)^T\vec{\alpha}^i\right)$ for $i=1,2,\ldots ,n.$ Then $(\matr{Q}+D+\sum_{i=1}^p\matr{G}^i, \vec{c}-\sum_{i=1}^pb_i\vec{\alpha}^i-diag(\matr{D}))$ is an equivalent representation of $(\matr{Q},\vec{c})$.
\end{cor}

It can be verified that the equivalent representations given by Corollary~\ref{cpa} are precisely of Billionnet et al~\cite{ba10} type. Following the ideas of Billionnet et al~\cite{ba10}, when $\matr{Q}$ is symmetric, the best values of $\vec{\alpha}^1,\vec{\alpha}^2,\ldots ,\vec{\alpha}^p$ and $\matr{D}$ (in terms of strong continuous relaxations) can be identified by solving an appropriate semidefinite program and its dual with suitable assumptions on the representation of $\mathscr{F}$. Further, by choosing $\vec{\alpha}^1,\vec{\alpha}^2,\ldots ,\vec{\alpha}^p$ as a basis  of the subspace of $R^n$ obtained by vectors satisfy CVP,  strong Billionnet et al~\cite{ba10} type  representation can be obtained. \\

Let us now examine some simple equivalent representations generated by various choices of $\matr{Y}$ and $\matr{U}$ in Theorem~\ref{thm-2.4} along with associated properties. Many of these representations are well known in the context of various special cases of QCOP. We summarize them below  with some elucidating remarks.

\begin{enumerate}

\item[1.] \textbf{Diagonal annihilation}: In Theorem~\ref{thm-2.4}, choose  $\matr{U}$ such that $u_{ii} = -q_{ii}$ for $i=1,\hdots,n$, and $\matr{Y}$ as the zero matrix. Then the diagonal elements of the resulting $\matr{Q}^{\prime}$ are zeros. We call this operation of constructing the equivalent representation $(\matr{Q}^{\prime},\vec{c}^{\prime})$ from $(\matr{Q},\vec{c})$  as \textit{diagonal annihilation}.

Although diagonal annihilation is a simple operation, some important properties of $\matr{Q}$ and $\matr{Q}^{\prime}$ could be very different. For example, the difference between the rank of these matrices could be arbitrarily  large. One matrix could be positive semidefinite while the other could be negative semidefinite or indefinite. The symmetry property, if exists, is preserved under this transformation.
If $q_{ii} \geq 0$ for all $i=1,2,\dots ,n$ then it can be verified that
$f_{R}(\matr{Q},\vec{c},\vec{x}) \leq f_{R}(\matr{Q}^{\prime},\vec{c}^{\prime},\vec{x})\mbox{ for all } \vec{x}\in \mathscr{F}_R$. Similarly, If $q_{ii} \leq 0$ for all $i=1,2,\dots ,n$ then
$f_{R}(\matr{Q},\vec{c},\vec{x}) \geq f_{R}(\matr{Q}^{\prime},\vec{c}^{\prime},\vec{x})\mbox{ for all } \vec{x}\in \mathscr{F}_R$.

 Since some of the crucial properties of $\matr{Q}$ could be altered by diagonal annihilation, the computational impact of this transformation  needs to be analyzed carefully.

\item[2.] \textbf{Linear term annihilation:} In Theorem~\ref{thm-2.4}, choose  $\matr{U}$ such that $u_{ii} = c_i$ for $i=1,\hdots,n$, and $\matr{Y}$ as the zero matrix. Then the resulting $\vec{c}^{\prime}$ is the zero vector of size $n$. We call this operation of constructing the equivalent representation $(\matr{Q}^{\prime},\vec{c}^{\prime})$ from $(\matr{Q},\vec{c})$  as \textit{linear term annihilation}.

    As in the case of diagonal annihilation, under this simple transformation, if $\matr{Q}$ is symmetric then  $\matr{Q}^{\prime}$ is also symmetric. However, properties such as rank, positive (negative) definiteness could be altered by the transformation. If $c_{i} \geq 0$ for all $i=1,2,\dots ,n$ then it can be verified that $f_{R}(\matr{Q},\vec{c},\vec{x}) \leq f_{R}(\matr{Q}^{\prime},\vec{c}^{\prime},\vec{x})\mbox{ for all } \vec{x}\in \mathscr{F}_R$. Similarly, If $c_{i} \leq 0$ for all $i=1,2,\dots ,n$ then
$f_{R}(\matr{Q},\vec{c},\vec{x}) \geq f_{R}(\matr{Q}^{\prime},\vec{c}^{\prime},\vec{x})\mbox{ for all } \vec{x}\in \mathscr{F}_R$.

Again the impact of this transformation on computational performance is not obvious and needs to abe analyzed carefully.

\item[3.] \textbf{Convexification:} In Theorem~\ref{thm-2.4}, choose  $\matr{U}$ such that $u_{ii} = M$ for $i=1,\hdots,n$ where $M$ is a nonnegative number, and $\matr{Y}$ as the zero matrix. We call this operation of constructing the equivalent representation $(\matr{Q}^{\prime},\vec{c}^{\prime})$ from $(\matr{Q},\vec{c})$  as \textit{convexification}.

     Note that by choosing $M$ sufficiently large, we can make $\matr{Q}^{\prime}$ positive semi-definite and hence the objective function of the continuous relaxation of the QCOP instance $(\matr{Q}^{\prime},\vec{c}^{\prime})$ is convex. In this case, it can be verified that $f_{R}(\matr{Q},\vec{c},\vec{x}) \geq f_{R}(\matr{Q}^{\prime},\vec{c}^{\prime},\vec{x})\mbox{ for all } \vec{x}\in \mathscr{F}_R$. However, we can solve the continuous relaxation of the instance $(\matr{Q}^{\prime},\vec{c}^{\prime})$ of QCOP in polynomial time whenever $\mathscr{F}$ have a compact representation using linear inequalities or an associated separation problem can be solved in polynomial time.  The transformation could alter the rank and for smaller values of $M$ it could affect properties such as positive semidefinite, negative semidefinite etc. of $\matr{Q}$, if existed.

     If $\matr{Q}$ is symmetric and not positive semidefinite, choosing $M$ to be the negative of its smallest eigenvalue is sufficient to make $\matr{Q}^{\prime}$ positive semidefinite~\cite{1970Hammer}. To choose such an $M$, additional computational effort is required. We also discussed earlier various other convexification strategies. However, by the transformation "convexification" we simply mean the simple operation indicated above by choosing $M$ sufficiently large.

\item[4.] \textbf{Concavification:} In Theorem~\ref{thm-2.4}, choose  $\matr{U}$ such that $u_{ii} = -M$ for $i=1,\hdots,n$ where $M$ is a nonnegative number, and $\matr{Y}$ as the zero matrix. We call this operation of constructing the equivalent representation $(\matr{Q}^{\prime},\vec{c}^{\prime})$ from $(\matr{Q},\vec{c})$  as \textit{concavification}.

     Note that by choosing $M$ sufficiently large, we can make $\matr{Q}^{\prime}$ negative semidefinite and hence the objective function of the continuous relaxation of the QCOP instance $(\matr{Q}^{\prime},\vec{c}^{\prime})$ is concave. In this case, QCOP is equivalent to its continuous relaxation. To see this, consider the quadratic programming problem (QPP($\mathscr{F}_R$))

\begin{align*}
\nonumber \mbox{~~ Minimize~ } & \sum_{i=1}^{n}\sum_{j=1}^{n} q_{ij} x_ix_j + \sum_{i=1}^{n} c_i x_i + M x_i(1-x_i)\\
\text{Subject to~~} & \vec{x} \in \mathscr{F}_R
\end{align*}

Since $M$ large, there always exist a binary optimal solution to QPP when $\mathscr{F}_R$ is a polyhedral set. Thus QPP($\mathscr{F}_R$) is equivalent to QPP($\mathscr{F}$) which is obtained from QPP($\mathscr{F}_R$) by replacing $\mathscr{F}_R$ with $\mathscr{F}$. Now replacing $x_i^2$ by $x_i$ (which is valid for binary variables) in the objective objective function of QPP($\mathscr{F}$) and simplifying, we get the instance  $(\matr{Q}^{\prime},\vec{c}^{\prime})$ of QCOP.

This observation shows that solving the continuous relaxation of $(\matr{Q}^{\prime},\vec{c}^{\prime})$ is as hard as solving QCOP.

\item[5.] \textbf{Triangularization:} In Theorem~\ref{thm-2.4}, choose  $\matr{U}$ such that $u_{ii} = -q_{ii}$ for $i=1,\hdots,n$  and $\matr{Y}$ such that
     \begin{equation*}
y_{ij} =
\begin{cases}
q_{ji} &\text{if  $j > i$}\\
-q_{ij} &\text{if $j < i$}\\
0 & \text{otherwise}.
\end{cases}
\end{equation*}
 Then the resulting matrix $\matr{Q}^{\prime}$ is upper triangular with its diagonal elements are zeros. We call this operation of constructing the equivalent representation $(\matr{Q}^{\prime},\vec{c}^{\prime})$ from $(\matr{Q},\vec{c})$  as \textit{triangularization}. Again, triangularization could affect rank and properties such as positive (negative) semidefiniteness, if exists for $\matr{Q}$.

\end{enumerate}

Applying the Theorems discusses above, in different combinations, many simple equivalent representations of a QCOP can be developed and studied. Since our focus in this paper is on computational effects of simple and commonly used  equivalent representations, to manage the study effectively, we restrict to ourselves to six equivalent transformations obtained by symmetrization, diagonal annihilation, linear term annihilation, convexification, concavification, and triangularization in appropriate combinations.

\subsection{Equivalent representations and the natural lower bound}

Effects of equivalent representations on lower bounds obtained by continuous relaxations of a QCOP was discussed in the previous subsection. Let us now consider another lower bound which is a generalization of the well known  Gilmore-Lawler \cite{1962Gilmore,1963Lwaler} lower bound for the quadratic assignment problem and its variations~\cite{1992Assad}. We call this bound a \textit{natural lower bound} for QCOP.

For $k=1,2,\ldots ,n$ let
 \begin{equation*}
 l_k = c_k + \min\left\{\sum_{j=1}^{n} q_{kj} x_j: \vec{x} \in \mathscr{F}, x_k=1\right\} \; \mbox{  and  }    \; m_k = c_k + \min\left\{\sum_{i=1}^{n} q_{ik} x_i : \vec{x} \in \mathscr{F}, x_k=1\right\}.
\end{equation*}

Also, let
 \begin{equation*}
\alpha = \min_{\vec{x} \in \mathscr{F}}    \sum_{k=1}^{n} l_k x_k   \; \; \; \mbox{  and }  \; \; \;  \beta = \min_{\vec{x} \in \mathscr{F}}    \sum_{k=1}^{n} m_k x_k.
\end{equation*}

\begin{thm}
 $\max\{\alpha,\beta\}$ is a lower bound for the optimal objective function value of QCOP.
\end{thm}

\begin{proof} Let $\vec{x}^0$ be an optimal solution to the QCOP instance $(\matr{Q},\vec{c})$ and $T = \{j : x_j^0 =1\}$. Then
\begin{align*}
f(\matr{Q},\vec{c},\vec{x}^0) &=\sum_{k\in T}\left(c_k + \sum_{j=1}^nq_{kj}x_j^0\right)
\geq \sum_{k\in T} l_k \geq \min\left\{\sum_{k=1}^nl_kx_j : \vec{x} \in \mathscr{F}\right\} = \alpha
\end{align*}
Similarly, one can show that $f(\matr{Q},\vec{c},\vec{x}^0) \geq \beta$ and the result follows.
\end{proof}

Note that each of the values $l_k,m_k$ for $k=1,2,\ldots ,n$ and $\alpha$ and $\beta$ can be identified by solving an associated \textit{linear combinatorial optimization problem} (LCOP). Thus, to identify the natural lower bound for a QCOP, we need to solve $2n+1$ LCOP. If the $\matr{Q}$ is symmetric, $\alpha = \beta$ and in this case, one need to solve only $n+1$ LCOP to identify the natural lower bound. For problems such as quadratic assignment or quadratic spanning tree, this LCOP can be solved efficiently in polynomial time. However, for some other examples such as the quadratic traveling salesman problem and the quadratic set covering problem, this LCOP itself is NP-hard. In such cases, one may be interested in using lower bounds on $l_k$ and $m_k$ and/or lower bounds for $\alpha$ and $\beta$. More specifically,

For $k=1,2,\ldots ,n$ let
 \begin{equation*}
 l^R_k = c_k + \min\left\{\sum_{j=1}^{n} q_{kj} x_j: \vec{x} \in \mathscr{F}_R, x_k=1\right\} \; \mbox{  and  }    \; m^R_k = c_k + \min\left\{\sum_{i=1}^{n} q_{ik} x_i : \vec{x} \in \mathscr{F}_R, x_k=1\right\}.
\end{equation*}

Also, let
 \begin{equation*}
\alpha^R = \min\left\{\sum_{k=1}^{n} \left\lceil l^R_k\right\rceil x_k : \vec{x} \in \mathscr{F}_R\right\}   \; \; \; \mbox{  and }  \; \; \;  \beta^R = \min\left\{ \sum_{k=1}^{n} \left\lceil m^R_k\right\rceil x_k : \vec{x} \in \mathscr{F}_R\right\}
\end{equation*}

and
 \begin{equation*}
\alpha^R_1 = \min\left\{\sum_{k=1}^{n}\left\lceil l^R_k\right\rceil x_k : \vec{x} \in \mathscr{F}\right\}  \; \; \; \mbox{  and }  \; \; \;  \beta^R_1 = \min\left\{\sum_{k=1}^{n} \left\lceil m^R_k\right\rceil x_k : \vec{x} \in \mathscr{F}\right\}.
\end{equation*}

\begin{thm}\label{thma-7}
$\max\{\alpha^R,\beta^R\}, \max\{\alpha^R_1,\beta^R_1\}$, and $ \max\{\alpha,\beta\}$ are lower bounds for the optimal objective function value of the QCOP. Further, $\max\{\alpha^R,\beta^R\}\leq \max\{\alpha^R_1,\beta^R_1\}\leq \max\{\alpha,\beta\}$.
\end{thm}

Note that the lower bound $\max\{\alpha^R,\beta^R\}$ can be identified in polynomial time since we are solving at most $2m+2$ linear programming problems under suitable assumptions on $\mathscr{F}_R$. The bound $\max\{\alpha^R_1,\beta^R_1\}$ is better than $\max\{\alpha^R,\beta^R\}$  but needs to solve $2m$ linear programs and two LCOPs.

\begin{cor}
The natural lower bound and its relaxations as discussed in Theorem~\ref{thma-7} obtained using any of the equivalent representations of a QCOP are lower bounds on the optimal objective function value of the QCOP.
\end{cor}

It is not difficult to construct examples where different equivalent representations of the same QCOP having different natural lower bound values. This makes it interesting to identify which equivalent representation is preferred in terms of obtaining stronger lower bounds. The effectiveness of these lower bounds and their relative computational benefits will be discussed in section 3.2.3.  Various extensions of the Gilmore-Lawler lower bound for the QAP are known  in literature \cite{1962Gilmore,1963Lwaler,1994Pardalos} For the sake of brevity, we do not study  them here and restrict  our experiments to the basic natural lower bound. \\

\section{ Computational Experiments} \label{qspdef}

In this section we present results of extensive computational experiments carried out using common and well known equivalent representations of QCOP. These include selected representations generated by  symmetrization, diagonal annihilation, linear term annihilation, convexification, concavification, and triangularization and their combinations. The quadratic set covering problem is used to generate test instances. We want to emphasize that our experimental  study considers representations that are commonly used  and those identified without significant computational efforts. Our goal is to identify a preferred representation for the general purpose quadratic 0-1 programming problem solvers within CPLEX and GUROBI  among such equivalent representations. Consequently, representations that require  solving semidefinite programs or equivalent Lagrangian problems are not considered in this experimental analysis. We used  the 0-1  quadratic programming solvers of CPLEX12.5  and GUROBI6.0.5 to solve the test instances. The programs are coded in C++ and tested on a PC with windows 7 operating system, Intel 4790 i7 3.60 GHz processor and 32 GB of RAM.  For CPLEX and GUROBI, time limit parameter is set to 3 hours and all other parameters  are set to their default values. For statistical analysis we use  non-parametric statistical test "Wilcoxon Signed Rank Test" \cite{1945Wilcoxon} with SPSS, a commercial statistical software. For all experiments, we use quadratic set covering instances as test problems. \\

 The \textit{quadratic set covering problem} (QSCP) can be defined as follows. Let $I = \{1,2,\dotsc,m\}$ be a finite set and  $P = \{P_1,P_2,\dots,P_n\}$  be a collection of subsets of $I$. Let $J= \{1,2 ,\dotsc, n\} $ be the index set of elements of $P$. For each element $j \in J$, a cost $c_j$ is assigned and for each element $(i,j) \in J \times J$, a cost $q_{ij}$ is also assigned. A subset $V$ of $J$ is  a {\it cover} of $I$, if {$\displaystyle \cup _{j \in V} P_j = I$}. The the quadratic set covering problem is to select a cover $L$ such that $\sum_{j\in L} c_j + \sum_{i\in L} \sum_{j\in L} q_{ij}$ is minimized. Choosing $E = J$ and $F$ as the family of all covers of $I$, QSCP can be viewed as a special case of QCOP. \\

Let $\matr{D}= (d_{ij})_{m \times n}$ be an $m \times n$ matrix and its \th{$(i,j)$} element $d_{ij}$ is given by
\begin{equation*}
d_{ij} =
\begin{cases}
1 &\text{if $i\in P_j$}\\
0 &\text{otherwise.}
\end{cases}
\end{equation*}

\noindent Also, consider the decision variables $x_1,x_2,\ldots ,x_n$ where

\begin{equation*}
x_j =
\begin{cases}
1 &\text{if set $P_j$ is selected}\\
0 &\text{otherwise.}
\end{cases}
\end{equation*}

The vector of decision variables is represented by $ \vec{x} =  (x_1, \hdots, x_n )^T$ and $\vec{1}$ is a column vector of size $m$ with all its entries are equal to 1. Then the QSCP can be formulated  as a 0-1 integer program
\begin{align*}
\nonumber \mbox{QSCP:~~ Minimize~ } & \vec{c}\vec{x} + \vec{x}^T\matr{Q}\vec{x}\\
 \mbox{Subject to } & \matr{D}\vec{x} \geq \vec{1},\\
 &\vec{x} \in \{0,1\}^n.
\end{align*}

The family of feasible solutions of  QSCP is denoted by $\Im  = \{ \vec{x}  | \matr{D}\vec{x}\ge \vec{1}, \vec{x} \in \{0,1\}^n \}$. The family of feasible solutions of the  continuous relaxation of QSCP  is given by $\bar{\Im } = \{ \vec{x}  | \matr{D}\vec{x}\ge \vec{1}, \vec{0} \leq \vec{x} \leq \vec{1} \}$.  \\

\subsection{Generation of QSCP test instances}

Some of our test instances are taken from the standard benchmark problems for the linear set covering problem \cite{balas,1987Beasley,wool} with a quadratic part added to the objective function. Many of these problems are of large size (for the quadratic objective function) and hence we selected only a small subset of them. The test instances selected from benchmark problems are scp41 ~\cite{balas},  scpe1, scpe2, scpe3, scpe4 and scpe5 ~\cite{1987Beasley}, and scpcyc06, scpcyc07, scpclr10, scpclr11, scpclr12, and  scpclr13~\cite{wool}. The remaining instances are generated randomly with feasibility guaranteed.  For these problems the naming convention qsc-m\#n\# is used, where \# represent the corresponding value of $m$ and $n$. A pseudo code for generating these instances are given below.\\

\begin{algorithm}[H]
\caption{Pseudocode for set cover  instance generation}\label{euclid}
\begin{algorithmic}[1]

\State Input $m$ and $n$, where $m$ is the number of elements and $n$ is the number of subsets of $\{1,\hdots,m\}$.
\State Set subset $S_j = \emptyset$, $j=1,\hdots,n$.

\For{\texttt{$i \gets 1 \text{ to } m$ }}
            \State {  $k_i = $ a random integer between $[1,n/2];$ }
            \State select $k_i$ subsets $\{S_{i_1},\hdots,S_{i_{K_{i}}}\} \subseteq  \{S_1,\hdots,S_n\}$ randomly such that each contains element $i$.
                \For{\texttt{$l \gets 1 \text{ to } k_i$ }}
                 \State $S_{i_{l}} = S_{i_l} \cup \{i\}$.
	            \EndFor

	 \EndFor
 \State Output feasible set covering test instance.
\end{algorithmic}
\end{algorithm}

It is possible that the instances generated by the pseudo code above may contain empty subsets. Note that for the linear set covering problem, such sets can easily be removed by considering the sign of the corresponding cost. No such easy mechanism is available for the quadratic objective function and hence the possibility of empty sets are left intact.\\

 For each of the instances selected from the set covering benchmark problems and those generated by our pseudocode, eight different cost matrices $Q$, with various properties are considered. These are summarized in Table~\ref{q1}.

\begin{table}[H]
\caption{Classes of quadratic cost matrices used in the experiments}
\centering
\resizebox{\columnwidth}{!}{%
\begin{tabular}{@{} c p{5.5cm} p{10.0cm}@{}} \toprule
Number & Property of $\matr{Q}$ & Method of generation\\
 \midrule

1  & Non-negative elements & Uniformly distributed random integers in the range [5,10] \\
2  & Positive semidefinite & A random matrix $\matr{B}$ is generated with elements uniformly distributed in the range [-5,5] and $\matr{Q}$ is set to $\matr{B}\matr{B}^T$. \\
3 & Non-negative and positive semidefinite & same as above except that the elements of $B$ are uniformly distributed random integers in the range [5,10].\\
4 & Random and balanced & Uniformly distributed random integers in the range [-5,5].\\
5 & Random and positively skewed & Uniformly distributed random integers in the range [-5,10].\\
6 & Random and negatively skewed & Uniformly distributed random integers in the range [-10,5]. \\
7 & Rank 1 & $\matr{Q}$ is set to $\vec{a}\vec{b}$ where $\vec{a}$ is a column vector and $\vec{b}$ is a row vector having elements as uniformly distributed random integers in the ranges [-10,10] and [-5,5] respectively.\\
8 & Rank 2 &  $\matr{Q}$ is set to $\vec{a}_1\vec{b}_1 + \vec{a}_2\vec{b}_2$ where $\vec{a}_1, \vec{a}_2$ are column vectors and $\vec{b}_1,\vec{b}_2$ are row vectors with  uniformly distributed random integers in ranges [-10,10] and [-5,5] respectively as elements.\\
\bottomrule
\end{tabular}
}
\label{q1}
\end{table}

For the benchmark instances taken from the literature, the cost vector $\vec{c}$ is chosen as the corresponding linear cost vectors of those instances. For instances that we generated, the linear cost vector is selected as the all one vector. The seed for the random number generator is set to $2n+3m+11$.
We use the notations ORG, SYM, CNX, CNV, UT, SYMI respectively for the  original ($\matr{Q},\vec{c}$), and the equivalent representations of $(\matr{Q},\vec{c})$ obtained by symmetrization, convexification, concavification, triangularization, and symmetrization followed by convexification.  For statistical analysis of our experimental data we use the non-parametric "Wilcoxon Signed Rank Test" \cite{1945Wilcoxon} using SPSS, a commercial statistical software package. Since $\matr{Q}$ is generated with some random criteria, ORG can also be viewed as a random equivalent perturbation applied on SYM or any other structures forms, using Theorem~\ref{thm-2.4}.

\subsection{Computational experiments and analysis of the results} Using the problem instances generated, we  carried out extensive computational experiments with the following objectives:
\begin{enumerate}
\item Compare the effect of the equivalent representations ORG, SYM, CNX, CNV, UT, and SYMI when the 0-1 quadratic programming solvers of  CPLEX and GUROBI are used as exact solvers on each class of test problems.
\item Compare the effect of the equivalent representations ORG, SYM, CNX, CNV, UT, and SYMI when the 0-1 quadratic programming solvers of  CPLEX and GUROBI are used as heuristic solvers (time restricted) on each class of test problems.
\item Effect of equivalent representations on the natural lower bound and its relaxations.
\end{enumerate}

\subsubsection{Effect of equivalent representations using CPLEX and GUROBI as exact solvers}
Our computational experiments identified  interesting patterns that can be used to provide additional guidelines on using CPLEX and GUROBI solvers. Let us now discuss the outcome of the experiments in details along with conclusions drawn.\\

 The results of the  experiments comparing the effect of equivalent representations on running time of CPLEX and GUROBI 0-1 quadratic programming solvers are summarized in tables \ref{table:QNNCPLEX-GUROBIM2}  - \ref{table:Rank2CPLEX-GUROBIM2}. Here the solvers are used as 'exact solver' to compute an optimal solution. In each table, there is a column corresponding to each of the equivalent representations considered and this column contains cpu time (in milliseconds) taken for the specific representation it is associated with and the specific solver used. The column `` Opt val" represents  the optimal objective function value.   There is only one column for this value since the optimal objective function values are same for all equivalent representations of the same instance.   The column ``problem'' represents the name of the instance used. In the column ``size'', $m$ is the number of constraints and $n$ is the number of variables.  The best CPU time is marked in boldface letter. Better CPU time between CPLEX and GUROBI  is identified by an asterisk(*). To manage the experiments effectively, a time limit of 3 hours was imposed on both solvers.  If a solver can not solve  an instance (representation) to  optimality with in 3 hours, we enter ``-" in place of CPU time. We also tried to solve many of the instances that are not solved in 3 hours by running without time limits and the problems were not solved after running overnight and we aborted such jobs. \\

 As the tables demonstrate, GUROBI is more sensitive towards different representations but solved faster the instances that are solved to optimality.  CPLEX appeared to be more robust with respect to the equivalent representations considered but was slower compared to GUROBI on most of the instances that GUROBI was able to solve. In general, for CPLEX, the equivalent representation obtained by symmetrization was more promising while for GUROBI,  triangularization produced a better representation. The statistical significance of this observation was confirmed using the Wilcoxin sign test \cite{1945Wilcoxon}. There are clear exceptions too. When $\matr{Q}$ is selected as positive semidefinite instances (Table \ref{table:QCNXCPLEX-GUROBIM2}), for the original representation, GUROBI uniformly produced superior running time  whereas for many of the other equivalent representations, the running time for this solver was prohibitively large. This observation is intriguing since symmetry with positive semidefinite is  generally viewed as a more  desirable property. Note that the representation ORG can be viewed as adding a random skew-symmetric matrix to SYM. Thus, it is not the case that more structured problems need not be the winner always. This points to the relevance of studying equivalent representations in general rather than more structured ones.  For this class, CPLEX did not achieve the impressive running times that GUROBI produced for any of the problems. However, we noticed that CPLEX was not much sensitive to the underlying representations for this class of problems.\\

 An overview of the outcomes of the above experiments are summarized in Table \ref{table:summary1} below followed by detailed results in Tables 3 to 10.\\

 \begin{table}[H]
\caption{Summary of results}
\centering
\resizebox{\columnwidth}{!}{%
\begin{tabular}{@{} l l l l @{}} \toprule[1.5pt]
Instance Class & CPLEX & GUROBI & better \\
 & best performing representation &   best performing representation & solver\\
& & & \\ \midrule

$\matr{Q}$ non-negative  & SYM   & UT & CPLEX \\
$\matr{Q}$ positive semidefinite  &   none & ORG & GUROBI  \\
$\matr{Q}$ non-negative and positive semidefinite & SYM   & ORG & CPLEX, GUROBI\\

$\matr{Q}$ arbitrary with symmetric distribution  &   SYM & CNV and UT  & CPLEX  \\
$\matr{Q}$ arbitrary with left-skewed distribution  &   SYM & UT  & CPLEX \\
$\matr{Q}$ arbitrary with right-skewed distribution  &   SYM and UT & UT  &  CPLEX \\
  & statistically no significant difference   &   \\

$\matr{Q}$ Rank 1  &   SYM and UT  & ORG and UT & GUROBI   \\
  & statistically no significant difference   & statistically UT is better than ORG &    \\
$\matr{Q}$ Rank 2  &   none  & UT  & GUROBI\\
\bottomrule[1.5pt]
\end{tabular}
}
\label{table:summary1}
\end{table}

\begin{landscape}
\begin{table}[H]
\small
\caption{Q  non-negative and in   [5,10], M=  10000}
\centering
\begin{tabular}{@{} l l c c |c c c c c c |c c c c c c @{}} \toprule[1.5pt]
  & &  & Opt & \multicolumn{5}{@{}c@{}}{CPLEX} & \multicolumn{5}{@{}c@{}}{GUROBI}  \\
\cmidrule[1.2pt](lr){5-10} \cmidrule[1.2pt](lr){11-16} \multirow{3}[3]{*}{problem}  & \multicolumn{2}{@{}c@{}}{size}& val & ORG & CNX & CNV & SYM & SYMI & UT & ORG & CNX & CNV & SYM & SYMI  & UT\\

 \cmidrule[1.2pt](lr){2-3}   &	\multirow{2}{*}{m} & \multirow{2}{*}{n}  &  & \multirow{2}{*}{time} & \multirow{2}{*}{time}  &\multirow{2}{*}{time}  &\multirow{2}{*}{time}  &\multirow{2}{*}{time} & \multirow{2}{*}{time} & \multirow{2}{*}{time}  &\multirow{2}{*}{time}  &\multirow{2}{*}{time}  &\multirow{2}{*}{time} &\multirow{2}{*}{time} &\multirow{2}{*}{time} \\
& & & & & & & & & & & & & \\ \midrule

qsc-m5n20	&	5	&	20	&	61	&	98	&	\textbf{73}	&	307	&	104	&	146	&	79	&	 53	&	63	&	59	&	 67	&	 125	&	\textbf{44}$^{*}$	\\
qsc-m5n30	&	5	&	30	&	27	&	70	&	67	&	232	&	69	&	91	&	\textbf{65}	&	 129	&	828	&	56	&	 106	&	125	&	\textbf{35}$^{*}$	\\
qsc-m10n20	&	10	&	20	&	115	&	\textbf{56}	&	\textbf{56}	&	61	&	58	&	81	&	 57	&	67	&	57	&	 33	&	 40	&	78	&	\textbf{20}$^{*}$	\\
qsc-m10n30	&	10	&	30	&	56	&	94	&	299	&	300	&	\textbf{80}	&	134	&	82	&	 44	&	821	&	55	&	 69	&	 172	&	\textbf{36}$^{*}$	\\
\end{tabular}
\label{table:QNNCPLEX-GUROBIM2}
\end{table}
\end{landscape}

\begin{landscape}
\begin{table}[H]
\small
\caption{Q  positive semidefinite and B in   [-5,5], M=  10000}
\centering
\begin{tabular}{@{} l l c c |c c c c c c |c c c c c c @{}} \toprule[1.5pt]
  & &  & Opt & \multicolumn{5}{@{}c@{}}{CPLEX} & \multicolumn{5}{@{}c@{}}{GUROBI}  \\
\cmidrule[1.2pt](lr){5-10} \cmidrule[1.2pt](lr){11-16} \multirow{3}[3]{*}{problem}  & \multicolumn{2}{@{}c@{}}{size}& val & ORG & CNX & CNV & SYM & SYMI & UT & ORG & CNX & CNV & SYM & SYMI  & UT\\

 \cmidrule[1.2pt](lr){2-3}   &	\multirow{2}{*}{m} & \multirow{2}{*}{n}  &  & \multirow{2}{*}{time} & \multirow{2}{*}{time}  &\multirow{2}{*}{time}  &\multirow{2}{*}{time}  &\multirow{2}{*}{time} & \multirow{2}{*}{time} & \multirow{2}{*}{time}  &\multirow{2}{*}{time}  &\multirow{2}{*}{time}  &\multirow{2}{*}{time} &\multirow{2}{*}{time} &\multirow{2}{*}{time} \\
& & & & & & & & & & & & & \\ \midrule

qsc-m5n20	&	5	&	20	&	4	&	272	&	181	&	365	&	\textbf{130}	&	178	&	 153	&	\textbf{45}$^{*}$	 &	 247	&	133	&	68	&	313	&	64\\	
qsc-m5n30	&	5	&	30	&	3	&	1279	&	1387	&	\textbf{1119}	&	1252	&	 1439	&	1285	&	 \textbf{46}$^{*}$	&	297079	&	1469	&	1426	&	454874	&	 1404\\	
qsc-m10n20	&	10	&	20	&	5	&	\textbf{101}	&	173	&	213	&	221	&	129	&	90	 &	\textbf{34}$^{*}$	 &	473	 &	58	&	62	&	438	&	49\\	
qsc-m10n30	&	10	&	30	&	3	&	596	&	598	&	\textbf{591}	&	684	&	685	&	 603	&	\textbf{47}$^{*}$	 &	 123532	&	3907	&	3399	&	62484	&	3414\\
\end{tabular}
\label{table:QCNXCPLEX-GUROBIM2}
\end{table}
\end{landscape}

\begin{landscape}
\begin{table}[H]
\small
\caption{Q non-negative and positive semidefinite, B in  [5,15], M=  10000}
\centering
\begin{tabular}{@{} l l c c |c c c c c c |c c c c c c @{}} \toprule[1.5pt]
  & &  & Opt & \multicolumn{5}{@{}c@{}}{CPLEX} & \multicolumn{5}{@{}c@{}}{GUROBI}  \\
\cmidrule[1.2pt](lr){5-10} \cmidrule[1.2pt](lr){11-16} \multirow{3}[3]{*}{problem}  & \multicolumn{2}{@{}c@{}}{size}& val & ORG & CNX & CNV & SYM & SYMI & UT & ORG & CNX & CNV & SYM & SYMI  & UT\\

 \cmidrule[1.2pt](lr){2-3}   &	\multirow{2}{*}{m} & \multirow{2}{*}{n}  &  & \multirow{2}{*}{time} & \multirow{2}{*}{time}  &\multirow{2}{*}{time}  &\multirow{2}{*}{time}  &\multirow{2}{*}{time} & \multirow{2}{*}{time} & \multirow{2}{*}{time}  &\multirow{2}{*}{time}  &\multirow{2}{*}{time}  &\multirow{2}{*}{time} &\multirow{2}{*}{time} &\multirow{2}{*}{time} \\
& & & & & & & & & & & & & \\ \midrule

qsc-m5n20	&	5	&	20	&	487	&	65	&	\textbf{59}	&	84	&	62	&	79	&	66	&	 53	&	60	&	 \textbf{32}$^{*}$	&	96	&	94	&	82\\	
qsc-m5n30	&	5	&	30	&	171	&	\textbf{51}	&	58	&	59	&	66	&	80	&	64	&	 \textbf{48}$^{*}$	&	 481	 &	67	&	58	&	125	&	53\\	
qsc-m10n20	&	10	&	20	&	788	&	53	&	\textbf{52}	&	53	&	55	&	76	&	54	&	 32	&	19	&	21	&	 39	&	 94	&	\textbf{20}$^{*}$	\\
qsc-m10n30	&	10	&	30	&	364	&	332	&	66	&	219	&	\textbf{66}	&	82	&	184	&	 66	&	295	&	 \textbf{41}$^{*}$	&	78	&	171	&	85\\
\end{tabular}
\label{table:QNN-CNXCPLEX-GUROBIM2}
\end{table}
\end{landscape}

\begin{landscape}
\begin{table}[H]
\small
\caption{Q arbitrary(symmetric distribution) is in [-5,5], M=  10000}
\centering
\begin{tabular}{@{} l l c c |c c c c c c |c c c c c c @{}} \toprule[1.5pt]
  & &  & Opt & \multicolumn{5}{@{}c@{}}{CPLEX} & \multicolumn{5}{@{}c@{}}{GUROBI}  \\
\cmidrule[1.2pt](lr){5-10} \cmidrule[1.2pt](lr){11-16} \multirow{3}[3]{*}{problem}  & \multicolumn{2}{@{}c@{}}{size}& val & ORG & CNX & CNV & SYM & SYMI & UT & ORG & CNX & CNV & SYM & SYMI  & UT\\

 \cmidrule[1.2pt](lr){2-3}   &	\multirow{2}{*}{m} & \multirow{2}{*}{n}  &  & \multirow{2}{*}{time} & \multirow{2}{*}{time}  &\multirow{2}{*}{time}  &\multirow{2}{*}{time}  &\multirow{2}{*}{time} & \multirow{2}{*}{time} & \multirow{2}{*}{time}  &\multirow{2}{*}{time}  &\multirow{2}{*}{time}  &\multirow{2}{*}{time} &\multirow{2}{*}{time} &\multirow{2}{*}{time} \\
& & & & & & & & & & & & & \\ \midrule

qsc-m5n20	&	5	&	20	&	-237	&	62	&	58	&	\textbf{56}	&	57	&	86	&	62	 &	36	&	45	&	 \textbf{28}$^{*}$	&	44	&	94	&	52\\	
qsc-m5n30	&	5	&	30	&	-606	&	\textbf{64}	&	225	&	298	&	115	&	145	&	 104	&	76	&	61	&	 \textbf{36}$^{*}$	&	66	&	125	&	\textbf{36}$^{*}$\\	
qsc-m10n20	&	10	&	20	&	-184	&	\textbf{59}	&	60	&	127	&	369	&	83	&	63	 &	36	&	30	&	 \textbf{25}$^{*}$	&	55	&	93	&	36\\	
qsc-m10n30	&	10	&	30	&	-567	&	\textbf{61}	&	65	&	62	&	277	&	269	&	78	 &	51	&	48	&	 \textbf{37}$^{*}$	&	84	&	141	&	76\\
\end{tabular}%
\label{table:QarbSymmCPLEX-GUROBIM2}
\end{table}
\end{landscape}

\begin{landscape}
\begin{table}[H]
\small
\caption{Q arbitrary(left-skewed distribution) is in  [-10,5], M=  10000}
\centering
\begin{tabular}{@{} l l c c |c c c c c c |c c c c c c @{}} \toprule[1.5pt]
  & &  & Opt & \multicolumn{5}{@{}c@{}}{CPLEX} & \multicolumn{5}{@{}c@{}}{GUROBI}  \\
\cmidrule[1.2pt](lr){5-10} \cmidrule[1.2pt](lr){11-16} \multirow{3}[3]{*}{problem}  & \multicolumn{2}{@{}c@{}}{size}& val & ORG & CNX & CNV & SYM & SYMI & UT & ORG & CNX & CNV & SYM & SYMI  & UT\\

 \cmidrule[1.2pt](lr){2-3}   &	\multirow{2}{*}{m} & \multirow{2}{*}{n}  &  & \multirow{2}{*}{time} & \multirow{2}{*}{time}  &\multirow{2}{*}{time}  &\multirow{2}{*}{time}  &\multirow{2}{*}{time} & \multirow{2}{*}{time} & \multirow{2}{*}{time}  &\multirow{2}{*}{time}  &\multirow{2}{*}{time}  &\multirow{2}{*}{time} &\multirow{2}{*}{time} &\multirow{2}{*}{time} \\
& & & & & & & & & & & & & \\ \midrule

qsc-m5n20	&	5	&	20	&	-1160	&	21	&	15	&	21	&	\textbf{14}$^{*}$	&	20	 &	17	&	 \textbf{33}	 &	 47	&	36	&	43	&	94	&	36\\	
qsc-m5n30	&	5	&	30	&	-2958	&	22	&	15	&	20	&	17	&	21	&	 \textbf{14}$^{*}$	&	38	&	46	 &	53	 &	65	&	109	&	\textbf{37}\\	
qsc-m10n20	&	10	&	20	&	-1180	&	14	&	23	&	\textbf{13}$^{*}$	&	14	&	19	 &	21	&	34	&	31	 &	38	 &	36	&	110	&	\textbf{20}\\	
qsc-m10n30	&	10	&	30	&	-2651	&	17	&	24	&	\textbf{16}$^{*}$	&	41	&	22	 &	45	&	38	&	31	 &	36	 &	66	&	109	&	\textbf{21}\\
\end{tabular}%
\label{table:QarbLSCPLEX-GUROBIM2}
\end{table}
\end{landscape}

\begin{landscape}
\begin{table}[H]
\small
\caption{Q arbitrary(right-skewed distribution) is in  [-5,10], M=  10000}
\centering
\begin{tabular}{@{} l l c c |c c c c c c |c c c c c c @{}} \toprule[1.5pt]
  & &  & Opt & \multicolumn{5}{@{}c@{}}{CPLEX} & \multicolumn{5}{@{}c@{}}{GUROBI}  \\
\cmidrule[1.2pt](lr){5-10} \cmidrule[1.2pt](lr){11-16} \multirow{3}[3]{*}{problem}  & \multicolumn{2}{@{}c@{}}{size}& val & ORG & CNX & CNV & SYM & SYMI & UT & ORG & CNX & CNV & SYM & SYMI  & UT\\

 \cmidrule[1.2pt](lr){2-3}   &	\multirow{2}{*}{m} & \multirow{2}{*}{n}  &  & \multirow{2}{*}{time} & \multirow{2}{*}{time}  &\multirow{2}{*}{time}  &\multirow{2}{*}{time}  &\multirow{2}{*}{time} & \multirow{2}{*}{time} & \multirow{2}{*}{time}  &\multirow{2}{*}{time}  &\multirow{2}{*}{time}  &\multirow{2}{*}{time} &\multirow{2}{*}{time} &\multirow{2}{*}{time} \\
& & & & & & & & & & & & & \\ \midrule

qsc-m5n20	&	5	&	20	&	-20	&	\textbf{67}	&	\textbf{67}	&	235	&	99	&	91	&	 68	&	77	&	40	&	 \textbf{33}$^{*}$	&	59	&	110	&	51\\	
qsc-m5n30	&	5	&	30	&	-34	&	\textbf{89}$^{*}$	&	120	&	104	&	133	&	262	&	 \textbf{89}$^{*}$	&	 3486	 &	3054	&	129	&	131	&	3549	&	\textbf{94}\\	
qsc-m10n20	&	10	&	20	&	13	&	\textbf{71}	&	79	&	313	&	\textbf{71}	&	95	&	 73	&	51	&	68	&	 30	&	 39	&	93	&	\textbf{21}$^{*}$\\	
qsc-m10n30	&	10	&	30	&	-24	&	\textbf{75}	&	\textbf{75}	&	231	&	112	&	124	&	 84	&	1526	&	 1514	 &	 78	&	107	&	1602	&	\textbf{69}$^{*}$\\
\end{tabular}%
\label{table:QarbRSCPLEX-GUROBIM2}
\end{table}
\end{landscape}


\begin{landscape}
\begin{table}[H]
\small
\caption{$\matr{Q}$ Rank 1, $\vec{a}$ in [-10,10] in $\vec{b}$ [-5,5], M=10000}
\centering
\begin{tabular}{@{} l l c c |c c c c c c |c c c c c c @{}} \toprule[1.5pt]
  & &  & Opt & \multicolumn{5}{@{}c@{}}{CPLEX} & \multicolumn{5}{@{}c@{}}{GUROBI}  \\
\cmidrule[1.2pt](lr){5-10} \cmidrule[1.2pt](lr){11-16} \multirow{3}[3]{*}{problem}  & \multicolumn{2}{@{}c@{}}{size}& val & ORG & CNX & CNV & SYM & SYMI & UT & ORG & CNX & CNV & SYM & SYMI  & UT\\

 \cmidrule[1.2pt](lr){2-3}   &	\multirow{2}{*}{m} & \multirow{2}{*}{n}  &  & \multirow{2}{*}{time} & \multirow{2}{*}{time}  &\multirow{2}{*}{time}  &\multirow{2}{*}{time}  &\multirow{2}{*}{time} & \multirow{2}{*}{time} & \multirow{2}{*}{time}  &\multirow{2}{*}{time}  &\multirow{2}{*}{time}  &\multirow{2}{*}{time} &\multirow{2}{*}{time} &\multirow{2}{*}{time} \\
& & & & & & & & & & & & & \\ \midrule

qsc-m5n20	&	5	&	20	&	-753	&	134	&	115	&	114	&	\textbf{105}	&	128	&	 129	&	\textbf{36}	&	 47	 &	44	&	53	&	94	&	37\\	
qsc-m5n30	&	5	&	30	&	-1439	&	412	&	511	&	405	&	\textbf{71}	&	106	&	80	 &	\textbf{40}$^{*}$	 &	75	 &	49	&	64	&	187	&	51\\	
qsc-m10n20	&	10	&	20	&	-454	&	147	&	92	&	\textbf{70}	&	187	&	209	&	 312	&	\textbf{24}$^{*}$	 &	 31	&	102	&	47	&	79	&	36\\	
qsc-m10n30	&	10	&	30	&	-1061	&	74	&	\textbf{71}	&	117	&	258	&	213	&	 133	&	47	&	158	&	 57	 &	 67	&	312	&	\textbf{36}$^{*}$\\
\end{tabular}
\label{table:Rank1CPLEX-GUROBIM2}
\end{table}
\end{landscape}

\begin{landscape}
\begin{table}[H]
\small
\caption{$\matr{Q}$ Rank 2, $\vec{a}_1, \vec{a}_2$ in [-10,10] in $\vec{b}_1, \vec{b}_2$ in [-5,5], M=10000}
\centering
\begin{tabular}{@{} l l c c |c c c c c c |c c c c c c @{}} \toprule[1.5pt]
  & &  & Opt & \multicolumn{5}{@{}c@{}}{CPLEX} & \multicolumn{5}{@{}c@{}}{GUROBI}  \\
\cmidrule[1.2pt](lr){5-10} \cmidrule[1.2pt](lr){11-16} \multirow{3}[3]{*}{problem}  & \multicolumn{2}{@{}c@{}}{size}& val & ORG & CNX & CNV & SYM & SYMI & UT & ORG & CNX & CNV & SYM & SYMI  & UT\\

 \cmidrule[1.2pt](lr){2-3}   &	\multirow{2}{*}{m} & \multirow{2}{*}{n}  &  & \multirow{2}{*}{time} & \multirow{2}{*}{time}  &\multirow{2}{*}{time}  &\multirow{2}{*}{time}  &\multirow{2}{*}{time} & \multirow{2}{*}{time} & \multirow{2}{*}{time}  &\multirow{2}{*}{time}  &\multirow{2}{*}{time}  &\multirow{2}{*}{time} &\multirow{2}{*}{time} &\multirow{2}{*}{time} \\
& & & & & & & & & & & & & \\ \midrule

qsc-m5n20	&	5	&	20	&	-690	&	89	&	72	&	62	&	\textbf{61}	&	81	&	62	 &	53	&	61	&	 \textbf{26}$^{*}$	&	45	&	64	&	34	\\
qsc-m5n30	&	5	&	30	&	-3345	&	277	&	76	&	121	&	\textbf{62}	&	112	&	 391	&	\textbf{36}$^{*}$	 &	 47	&	77	&	69	&	66	&	37	\\
qsc-m10n20	&	10	&	20	&	-1007	&	54	&	56	&	55	&	55	&	99	&	 \textbf{53}	&	70	&	64	&	 \textbf{31}$^{*}$	&	39	&	55	&	59	\\
qsc-m10n30	&	10	&	30	&	-786	&	262	&	132	&	77	&	76	&	301	&	 \textbf{74}$^{*}$	&	76	&	244	&	 \textbf{75}	&	78	&	280	&	81	\\
\end{tabular}
\label{table:Rank2CPLEX-GUROBIM2}
\end{table}
\end{landscape}

\clearpage

\subsubsection{Experimental results when  CPLEX and GUROBI are used as heuristic solvers} \label{sectionheu}

To assess the effectiveness of CPLEX and GUROBI as a heuristic solver and compare their relative sensitivity towards equivalent representations, we run each of the test instances by specifying two preset time limits of 15 minutes and 30 minutes. For small size problems, all representations produced optimal values within in 30 minutes, although for many cases, optimality was not proved. Thus, we discuss further only experimental results for moderately large instances. \\

In the tables   \ref{table:TL15-30CPLEX-NN} to \ref{table:TL15-30CPLEX-ARBRS}, the column "val" represents the heuristic solution value obtained by the specific representation  associated with and the specific solver with the specified time limit.   The best heuristic solution is marked in boldface letter.  \\

It may be noted that the difference between various representation  solutions quality is not significant different, but we did observe , some level of statistical significance as  summarize in  Table \ref{table:summaryheu}.\\\\

\begin{table}[H]
\caption{Summary of results: Heuristic value}
\centering
\resizebox{\columnwidth}{!}{%
\begin{tabular}{@{} l| l |l |l| l  @{}} \toprule[1.5pt]
 &  \multicolumn{4}{@{}c@{}}{Best performing representation}\\
Instance Class & \multicolumn{2}{@{}c@{}}{CPLEX} & \multicolumn{2}{@{}c@{}}{GUROBI}  \\
  & 30 minutes & 15 minutes & 30 minutes & 15 minutes  \\
& & & & \\ \midrule
$\matr{Q}$ non-negative  &  - &  CNX, CNV, SYMI & none& CNV, SYM, UT\\
$\matr{Q}$ positive semidefinite  & - & -  &  ORG & ORG   \\
$\matr{Q}$ non-negative and positive semidefinite & -  & -  & ORG  & ORG  \\
$\matr{Q}$ arbitrary with symmetric distribution  & all & all  & all & all    \\
$\matr{Q}$ arbitrary with left-skewed distribution  &  all & all & all & all    \\
$\matr{Q}$ arbitrary with right-skewed distribution  & ORG & ORG, UT &  CNV, SYM, UT & CNV, SYM, UT    \\
\bottomrule[1.5pt]
\end{tabular}
}
\label{table:summaryheu}
\end{table}

\begin{landscape}
\begin{table}[H]
\small
\caption{Q non-negative is in  [5,10], M=  10000, Time Limit 15 minutes and 30 minutes}
\centering
\resizebox{\columnwidth}{!}{%
\begin{tabular}{@{} l l c c |c c c c c c |c c c c c c@{}} \toprule[1.5pt]
  & &  & Opt & \multicolumn{5}{@{}c@{}}{30 Minutes} & \multicolumn{5}{@{}c@{}}{15 Minutes}  \\
\cmidrule[1.2pt](lr){5-10} \cmidrule[1.2pt](lr){11-16} \multirow{3}[3]{*}{problem}  & \multicolumn{2}{@{}c@{}}{size}& val & ORG & CNX & CNV & SYM & SYMI & UT & ORG & CNX & CNV & SYM  & SYMI & UT\\

 \cmidrule[1.2pt](lr){2-3}   &	\multirow{2}{*}{m} & \multirow{2}{*}{n}  &  & \multirow{2}{*}{val} & \multirow{2}{*}{val}  &\multirow{2}{*}{val}  &\multirow{2}{*}{val}  &\multirow{2}{*}{val} & \multirow{2}{*}{val} & \multirow{2}{*}{val}  &\multirow{2}{*}{val}  &\multirow{2}{*}{val}  &\multirow{2}{*}{val}  &\multirow{2}{*}{val}  &\multirow{2}{*}{val}\\
& & & & & & & & & & & & & & & \\ \midrule

\multicolumn{16}{@{}c@{}}{CPLEX}\\ \cmidrule[1.2pt](lr){1-16}

qsc-m50n100	&	50	&	100	&	\textbf{331}	&	\textbf{331}	&	\textbf{331}	&	 \textbf{331}	&	 \textbf{331}	 &	\textbf{331}	&	\textbf{331}	&	\textbf{331}	 &	\textbf{331}	&	\textbf{331}	&	 \textbf{331}	&	 \textbf{331}	&	 \textbf{331}\\
qsc-m50n200	&	50	&	200	&	340	&	341	&	\textbf{340}	&	\textbf{340}	&	341	&	 \textbf{340}	&	 341	 &	 341	&	\textbf{340}	&	\textbf{340}	&	341	&	 \textbf{340}	&	341\\
qsc-m50n300	&	50	&	300	&	265	&	\textbf{265}	&	\textbf{265}	&	\textbf{265}	 &	\textbf{265}	&	 \textbf{265}	&	\textbf{265}	&	\textbf{265}	&	 \textbf{265}	&	\textbf{265}	&	\textbf{265}	&	 \textbf{265}	&	 \textbf{265}\\
\end{tabular}%
 }
  \label{table:TL15-30CPLEX-NN}
\end{table}
\end{landscape}

\begin{landscape}
\begin{table}[H]
\small
\caption{Q PSD  [-5,5], M=  10000, Time Limit 15 minutes and 30 minutes}
\centering
\begin{tabular}{@{} l l c c |c c c c c c |c c c c c c@{}} \toprule[1.5pt]
  & &  & Opt & \multicolumn{5}{@{}c@{}}{30 Minutes} & \multicolumn{5}{@{}c@{}}{15 Minutes}  \\
\cmidrule[1.2pt](lr){5-10} \cmidrule[1.2pt](lr){11-16} \multirow{3}[3]{*}{problem}  & \multicolumn{2}{@{}c@{}}{size}& val & ORG & CNX & CNV & SYM & SYMI & UT & ORG & CNX & CNV & SYM  & SYMI & UT\\

 \cmidrule[1.2pt](lr){2-3}   &	\multirow{2}{*}{m} & \multirow{2}{*}{n}  &  & \multirow{2}{*}{val} & \multirow{2}{*}{val}  &\multirow{2}{*}{val}  &\multirow{2}{*}{val}  &\multirow{2}{*}{val} & \multirow{2}{*}{val} & \multirow{2}{*}{val}  &\multirow{2}{*}{val}  &\multirow{2}{*}{val}  &\multirow{2}{*}{val}  &\multirow{2}{*}{val}  &\multirow{2}{*}{val}\\
& & & & & & & & & & & & & & & \\ \midrule

\multicolumn{16}{@{}c@{}}{CPLEX}\\ \cmidrule[1.2pt](lr){1-16}
qsc-m50n100	&	50	&	100	&	8	&	\textbf{8}	&	\textbf{8}	&	\textbf{8}	&	 \textbf{8}	&	\textbf{8}	&	 \textbf{8}	&	\textbf{8}	&	\textbf{8}	&	\textbf{8}	&	 \textbf{8}	&	\textbf{8}	&	\textbf{8}	\\
qsc-m50n200	&	50	&	200	&	8	&	\textbf{8}	&	\textbf{8}	&	\textbf{8}	&	 \textbf{8}	&	\textbf{8}	&	 \textbf{8}	&	\textbf{8}	&	\textbf{8}	&	\textbf{8}	&	 \textbf{8}	&	\textbf{8}	&	\textbf{8}	\\
qsc-m50n300	&	50	&	300	&	7	&	\textbf{7}	&	\textbf{7}	&	\textbf{7}	&	 \textbf{7}	&	\textbf{7}	&	 \textbf{7}	&	\textbf{8}	&	10	&	\textbf{8}	&	 \textbf{8}	&	\textbf{8}	&	\textbf{8}	\\
\end{tabular}%
  \label{table:TL15-30CPLEX-PSD}
\end{table}
\end{landscape}

\begin{landscape}
\begin{table}[H]
\small
\caption{Q NN PSD [5,15], M=  10000, Time Limit 15 minutes and 30 minutes}
\centering
\resizebox{\columnwidth}{!}{%
\begin{tabular}{@{} l l c c |c c c c c c |c c c c c c@{}} \toprule[1.5pt]
  & &  & Opt & \multicolumn{5}{@{}c@{}}{30 Minutes} & \multicolumn{5}{@{}c@{}}{15 Minutes}  \\
\cmidrule[1.2pt](lr){5-10} \cmidrule[1.2pt](lr){11-16} \multirow{3}[3]{*}{problem}  & \multicolumn{2}{@{}c@{}}{size}& val & ORG & CNX & CNV & SYM & SYMI & UT & ORG & CNX & CNV & SYM  & SYMI & UT\\

 \cmidrule[1.2pt](lr){2-3}   &	\multirow{2}{*}{m} & \multirow{2}{*}{n}  &  & \multirow{2}{*}{val} & \multirow{2}{*}{val}  &\multirow{2}{*}{val}  &\multirow{2}{*}{val}  &\multirow{2}{*}{val} & \multirow{2}{*}{val} & \multirow{2}{*}{val}  &\multirow{2}{*}{val}  &\multirow{2}{*}{val}  &\multirow{2}{*}{val}  &\multirow{2}{*}{val}  &\multirow{2}{*}{val}\\
& & & & & & & & & & & & & & & \\ \midrule

\multicolumn{16}{@{}c@{}}{CPLEX}\\ \cmidrule[1.2pt](lr){1-16}
qsc-m50n100	&	50	&	100	&	3143	&	\textbf{3143}	&	\textbf{3143}	&	 \textbf{3143}	&	\textbf{3143}	 &	 \textbf{3143}	&	\textbf{3143}	&	\textbf{3143}	 &	\textbf{3143}	&	\textbf{3143}	&	\textbf{3143}	 &	 \textbf{3143}	&	 \textbf{3143}	\\
qsc-m50n200	&	50	&	200	&	2712	&	\textbf{2712}	&	\textbf{2712}	&	 \textbf{2712}	&	\textbf{2712}	 &	 \textbf{2712}	&	\textbf{2712}	&	\textbf{2712}	 &	\textbf{2712}	&	\textbf{2712}	&	\textbf{2712}	 &	 \textbf{2712}	&	 \textbf{2712}	\\
qsc-m50n300	&	50	&	300	&	1856	&	\textbf{1856}	&	\textbf{1856}		&	 \textbf{1856}		&	 \textbf{1856}		 &	\textbf{1856}		&	\textbf{1856}		&	 \textbf{1856}		&	\textbf{1856}		&	 \textbf{1856}		&	 \textbf{1856}		&	 \textbf{1856}		&	\textbf{1856}		\\
\end{tabular}%
 }
  \label{table:TL15-30CPLEX-NNPSD}
\end{table}
\end{landscape}

\begin{landscape}
\begin{table}[H]
\small
\caption{Q ARB SYM  [-5,5], M=  10000, Time Limit 15 minutes and 30 minutes}
\centering
\resizebox{\columnwidth}{!}{%
\begin{tabular}{@{} l l c c |c c c c c c |c c c c c c@{}} \toprule[1.5pt]
  & &  & Opt & \multicolumn{5}{@{}c@{}}{30 Minutes} & \multicolumn{5}{@{}c@{}}{15 Minutes}  \\
\cmidrule[1.2pt](lr){5-10} \cmidrule[1.2pt](lr){11-16} \multirow{3}[3]{*}{problem}  & \multicolumn{2}{@{}c@{}}{size}& val & ORG & CNX & CNV & SYM & SYMI & UT & ORG & CNX & CNV & SYM  & SYMI & UT\\

 \cmidrule[1.2pt](lr){2-3}   &	\multirow{2}{*}{m} & \multirow{2}{*}{n}  &  & \multirow{2}{*}{val} & \multirow{2}{*}{val}  &\multirow{2}{*}{val}  &\multirow{2}{*}{val}  &\multirow{2}{*}{val} & \multirow{2}{*}{val} & \multirow{2}{*}{val}  &\multirow{2}{*}{val}  &\multirow{2}{*}{val}  &\multirow{2}{*}{val}  &\multirow{2}{*}{val}  &\multirow{2}{*}{val}\\
& & & & & & & & & & & & & & & \\ \midrule

\multicolumn{16}{@{}c@{}}{CPLEX}\\ \cmidrule[1.2pt](lr){1-16}
qsc-m50n100	&	50	&	100	&	-4997	&	-4997	&	-4997	&	-4997	&	-4997	&	 -4997	&	-4997	&	 -4997	&	 -4997	&	-4997	&	-4997	&	-4997	&	-4997	 \\
qsc-m50n200	&	50	&	200	&	-20368	&	-20368	&	-20368	&	-20368	&	-20368	&	 -20368	&	-20368	&	 -20368	&	 -20368	&	-20368	&	-20368	&	-20368	&	-20368	 \\
qsc-m50n300	&	50	&	300	&	-43875	&	-43875	&	-43875	&	-43875	&	-43875	&	 -43875	&	-43875	&	 -43875	&	 -43875	&	-43875	&	-43875	&	-43875	&	-43875	 \\
\end{tabular}%
 }
  \label{table:TL15-30CPLEX-ARBSYM}
\end{table}
\end{landscape}

\begin{landscape}
\begin{table}[H]
\small
\caption{Q ARB LS  [-10,5], M=  10000, Time Limit 15 minutes and 30 minutes}
\centering
\resizebox{\columnwidth}{!}{%
\begin{tabular}{@{} l l c c |c c c c c c |c c c c c c@{}} \toprule[1.5pt]
  & &  & Opt & \multicolumn{5}{@{}c@{}}{30 Minutes} & \multicolumn{5}{@{}c@{}}{15 Minutes}  \\
\cmidrule[1.2pt](lr){5-10} \cmidrule[1.2pt](lr){11-16} \multirow{3}[3]{*}{problem}  & \multicolumn{2}{@{}c@{}}{size}& val & ORG & CNX & CNV & SYM & SYMI & UT & ORG & CNX & CNV & SYM  & SYMI & UT\\

 \cmidrule[1.2pt](lr){2-3}   &	\multirow{2}{*}{m} & \multirow{2}{*}{n}  &  & \multirow{2}{*}{val} & \multirow{2}{*}{val}  &\multirow{2}{*}{val}  &\multirow{2}{*}{val}  &\multirow{2}{*}{val} & \multirow{2}{*}{val} & \multirow{2}{*}{val}  &\multirow{2}{*}{val}  &\multirow{2}{*}{val}  &\multirow{2}{*}{val}  &\multirow{2}{*}{val}  &\multirow{2}{*}{val}\\
& & & & & & & & & & & & & & & \\ \midrule

\multicolumn{16}{@{}c@{}}{CPLEX}\\ \cmidrule[1.2pt](lr){1-16}

qsc-m50n100	&	50	&	100	&	-30166	&	-30166	&	-30166	&	-30166	&	-30166	&	 -30166	&	-30166	&	 -30166	&	 -30166	&	-30166	&	-30166	&	-30166	&	-30166	 \\
qsc-m50n200	&	50	&	200	&	-119683	&	-119683	&	-119683	&	-119683	&	-119683	&	 -119683	&	-119683	&	 -119683	&	-119683	&	-119683	&	-119683	&	-119683	&	 -119683	\\
qsc-m50n300	&	50	&	300	&	-271725	&	-271725	&	-271725	&	-271725	&	-271725	&	 -271725	&	-271725	&	 -271725	&	-271725	&	-271725	&	-271725	&	-271725	&	 -271725	\\
\end{tabular}%
 }
  \label{table:TL15-30CPLEX-ARBLS}
\end{table}
\end{landscape}

\begin{landscape}
\begin{table}[H]
\small
\caption{Q ARB RS  [-5,10], M=  10000, Time Limit 15 minutes and 30 minutes}
\centering
\resizebox{\columnwidth}{!}{%
\begin{tabular}{@{} l l c c |c c c c c c |c c c c c c@{}} \toprule[1.5pt]
  & &  & Opt & \multicolumn{5}{@{}c@{}}{30 Minutes} & \multicolumn{5}{@{}c@{}}{15 Minutes}  \\
\cmidrule[1.2pt](lr){5-10} \cmidrule[1.2pt](lr){11-16} \multirow{3}[3]{*}{problem}  & \multicolumn{2}{@{}c@{}}{size}& val & ORG & CNX & CNV & SYM & SYMI & UT & ORG & CNX & CNV & SYM  & SYMI & UT\\

 \cmidrule[1.2pt](lr){2-3}   &	\multirow{2}{*}{m} & \multirow{2}{*}{n}  &  & \multirow{2}{*}{val} & \multirow{2}{*}{val}  &\multirow{2}{*}{val}  &\multirow{2}{*}{val}  &\multirow{2}{*}{val} & \multirow{2}{*}{val} & \multirow{2}{*}{val}  &\multirow{2}{*}{val}  &\multirow{2}{*}{val}  &\multirow{2}{*}{val}  &\multirow{2}{*}{val}  &\multirow{2}{*}{val}\\
& & & & & & & & & & & & & & & \\ \midrule

\multicolumn{16}{@{}c@{}}{CPLEX}\\ \cmidrule[1.2pt](lr){1-16}
qsc-m50n100	&	50	&	100	&	-	&	\textbf{0}	&	\textbf{0}	&	\textbf{0}	&	 \textbf{0}	&	\textbf{0}	&	 \textbf{0}	&	\textbf{0}	&	\textbf{0}	&	\textbf{0}	&	 \textbf{0}	&	\textbf{0}	&	\textbf{0}	\\
qsc-m50n200	&	50	&	200	&	-	&	-34	&	-32	&	-32	&	\textbf{-35}	&	-32	&	 \textbf{-35}	&	-34	&	 -29	 &	\textbf{-35}	&	\textbf{-35}	&	-29	&	 \textbf{-35}	\\
qsc-m50n300	&	50	&	300	&	-	&	\textbf{-77}	&	\textbf{-77}	 &	\textbf{-77}	 &	-48	&	 \textbf{-77}	 &	 -48	&	\textbf{-77}	&	\textbf{-77}	&	\textbf{-77}	 &	-48	&	\textbf{-77}	&	 -48	\\
\end{tabular}%
 }
  \label{table:TL15-30CPLEX-ARBRS}
\end{table}
\end{landscape}

\begin{landscape}
\begin{table}[H]
\small
\caption{Q Rank 1  , M=  10000, Time Limit 15 minutes and 30 minutes}
\centering
\resizebox{\columnwidth}{!}{%
\begin{tabular}{@{} l l c c |c c c c c c |c c c c c c@{}} \toprule[1.5pt]
  & &  & Opt & \multicolumn{5}{@{}c@{}}{30 Minutes} & \multicolumn{5}{@{}c@{}}{15 Minutes}  \\
\cmidrule[1.2pt](lr){5-10} \cmidrule[1.2pt](lr){11-16} \multirow{3}[3]{*}{problem}  & \multicolumn{2}{@{}c@{}}{size}& Sol & ORIG & CNX & CNV & SYM & SYMI & UT & ORIG & CNX & CNV & SYM  & SYMI & UT\\

 \cmidrule[1.2pt](lr){2-3}   &	\multirow{2}{*}{m} & \multirow{2}{*}{n}  &  & \multirow{2}{*}{Sol} & \multirow{2}{*}{Sol}  &\multirow{2}{*}{Sol}  &\multirow{2}{*}{Sol}  &\multirow{2}{*}{Sol} & \multirow{2}{*}{Sol} & \multirow{2}{*}{Sol}  &\multirow{2}{*}{Sol}  &\multirow{2}{*}{Sol}  &\multirow{2}{*}{Sol}  &\multirow{2}{*}{Sol}  &\multirow{2}{*}{Sol}\\
& & & & & & & & & & & & & & & \\ \midrule

\multicolumn{16}{@{}c@{}}{CPLEX}\\ \cmidrule[1.2pt](lr){1-16}
qsc-m50n100	&	50	&	100	&	0	&	-15640	&	-15640	&	-15640	&	-15640	&	-15640	 &	-15640	&	-15640	 &	 -15640	&	-15640	&	-15640	&	-15640	&	-15640	\\
qsc-m50n200	&	50	&	200	&	0	&	-82319	&	-82319	&	-82319	&	-82319	&	-82319	 &	-82319	&	-82319	 &	 -82319	&	-82319	&	-82319	&	-82319	&	-82319	\\
qsc-m50n300	&	50	&	300	&	0	&	-119160	&	-119160	&	-119160	&	-119160	&	 -119160	&	-119160	&	 -119160	 &	-119160	&	-119160	&	-119160	&	-119160	&	 -119160	\\
\end{tabular}%
 }
  \label{table:TL15-30CPLEX-Rank1}
\end{table}
\end{landscape}

\begin{landscape}
\begin{table}[H]
\small
\caption{Q Rank2, M=  10000, Time Limit 15 minutes and 30 minutes}
\centering
\resizebox{\columnwidth}{!}{%
\begin{tabular}{@{} l l c c |c c c c c c |c c c c c c@{}} \toprule[1.5pt]
  & &  & Opt & \multicolumn{5}{@{}c@{}}{30 Minutes} & \multicolumn{5}{@{}c@{}}{15 Minutes}  \\
\cmidrule[1.2pt](lr){5-10} \cmidrule[1.2pt](lr){11-16} \multirow{3}[3]{*}{problem}  & \multicolumn{2}{@{}c@{}}{size}& Sol & ORIG & CNX & CNV & SYM & SYMI & UT & ORIG & CNX & CNV & SYM  & SYMI & UT\\

 \cmidrule[1.2pt](lr){2-3}   &	\multirow{2}{*}{m} & \multirow{2}{*}{n}  &  & \multirow{2}{*}{Sol} & \multirow{2}{*}{Sol}  &\multirow{2}{*}{Sol}  &\multirow{2}{*}{Sol}  &\multirow{2}{*}{Sol} & \multirow{2}{*}{Sol} & \multirow{2}{*}{Sol}  &\multirow{2}{*}{Sol}  &\multirow{2}{*}{Sol}  &\multirow{2}{*}{Sol}  &\multirow{2}{*}{Sol}  &\multirow{2}{*}{Sol}\\
& & & & & & & & & & & & & & & \\ \midrule

\multicolumn{16}{@{}c@{}}{CPLEX}\\ \cmidrule[1.2pt](lr){1-16}
qsc-m50n100	&	50	&	100	&	-17694	&	-17694	&	-17694	&	-17694	&	-17694	&	-17694	 &	-17694	&	-17694	 &	 -17694	&	-17694	&	-17694	&	-17694	&	-17694	\\
qsc-m50n200	&	50	&	200	&	-50658	&	-50658	&	-50658	&	-50658	&	-50658	&	-50658	 &	-50658	&	-50658	 &	 -50658	&	-50658	&	-50658	&	-50658	&	-50658	\\
qsc-m50n300	&	50	&	300	&	-127866	&	-127866	&	-127866	&	-127867	&	-127866	&	 -127866	&	-127866	&	 -127866	&	-127866	&	-127867	&	-127866	&	-127866	&	 -127866	\\
\end{tabular}%
 }
  \label{table:TL15-30CPLEX-Rank2}
\end{table}
\end{landscape}


\clearpage

\subsubsection{Experiment results for  equivalent representations and the natural lower bound} \label{section4}

Although equivalent representations preserve optimal objective function value, these transformations need not preserve the natural lower bound value. In this subsection, we compare the effect of the representations ORG,CNX,CNV,SYM, SYMI, and UT on the natural lower bound and its relaxations.

We denote the lower bound $\max\{\alpha^R,\beta^R\}$ by NLB-R, $\max\{\alpha^R_1,\beta^R_1\}$ by NLP-R1, and $ \max\{\alpha,\beta\}$ by NLB.  For  each test instance  three natural lower bounds mentioned above  are identical for each representation.  CPLEX is used in all experiments as the solver. A summary of the results are presented in Table  in Table \ref{table:NLBfrequency}. The column " frequency" represents number of instances for which a representation obtained the tighter lower bound  out of 32 instances we consider.  As the table shows, most of the time SYM and SYMI representations produced  better NLB bound,  NLBRLX bound, and INLBRelax bound  compare to other representations. There are some exceptions too. When $\matr{Q}$ is selected as non negative instances (Table \ref{NLBtime}), for NLB bound,  ORG, CNX, and CNV representations attained tighter lower bounds for 24 instance but  SYM  and SYMI representations only 16 times.

\begin{table}[H]
\small
\caption{Frequency table for tighter NLB bounds (32 test instances for each class) }
\centering
\resizebox{\columnwidth}{!}{%
\begin{tabular}{@{} l l l l l  l l  | l   l l l   l l | l l  l l l l @{}} \toprule[1.5pt]
\multirow{3}[3]{*}{Instance class}  & \multicolumn{15}{@{}c@{}}{Frequency } \\
 \cmidrule[1.2pt](lr){2-19} & \multicolumn{5}{@{}c@{}}{NLB bound} &  \multicolumn{5}{@{}c@{}}{NLBRelax bound} &  \multicolumn{5}{@{}c@{}}{INLBRelax bound}   \\
 \cmidrule[1.2pt](lr){2-7}	\cmidrule[1.2pt](lr){8-13} \cmidrule[1.2pt](lr){14-19}  &   ORG & CNX & CNV & SYM & SYMI & UT  &   ORG & CNX & CNV & SYM & SYMI & UT &   ORG & CNX & CNV & SYM & SYMI & UT\\
& & & & & & & & & & & & & & & & & & \\ \midrule

$\matr{Q}$ non-negative & 24 & 24 & 24 & 16 & 16 & 0 & 22 & 22 & 22 & 13 & 13 & 0 & 24 & 24 & 24 & 12 & 12 & 0  \\
$\matr{Q}$ positive semidefinite  & 29 & 29 & 29 & 29 & 29 & 3 & 29 & 29 & 29 & 29 & 29 & 3 & 29 & 29 & 29 & 29 & 29 & 3 \\
$\matr{Q}$ non-negative and positive semidefinite &  32 & 32 & 32 & 32 & 32 &  0 &  32 & 32 & 32 & 32 & 32 &  0 &  32 & 32 & 32 & 32 & 32 &  0 \\

$\matr{Q}$ arbitrary with symmetric distribution  & 0 & 0 & 0 & 32 & 32 & 0  & 0 & 0 & 0 & 32 & 32 & 0  & 0 & 0 & 0 & 32 & 32 & 0   \\
$\matr{Q}$ arbitrary with left-skewed distribution  & 0 & 0 & 0 & 32 & 32 & 0  & 0 & 0 & 0 & 32 & 32 & 0  & 0 & 0 & 0 & 32 & 32 & 0   \\
$\matr{Q}$ arbitrary with right-skewed distribution  & 0 & 0 & 0 & 32 & 32 & 0  & 0 & 0 & 0 & 32 & 32 & 0  & 0 & 0 & 0 & 32 & 32 & 0   \\
$\matr{Q}$ Rank 1  &   0 & 0 & 0 & 31 & 31 & 1 & 0 & 0 & 0 & 31 & 31 & 1 & 0 & 0 & 0 & 31 & 31 & 1    \\
$\matr{Q}$ Rank 2  &    0 & 0 & 0 & 31 & 31 & 1 & 0 & 0 & 0 & 31 & 31 & 1 & 0 & 0 & 0 & 31 & 31 & 1    \\
\bottomrule[1.5pt]
\end{tabular}
}
\label{table:NLBfrequency}
\end{table}

In tables  \ref{NLBtime} - \ref{NLBtime2} we collect the minimum CPU time, the maximum CPU time, and the average CPU time used to obtained the natural lower bound are presented. The CPU time is  in milliseconds.  From the experimental results, we observe that for most of the class of problems, SYM representation uses the minimum CPU time to compute NLB, and NLBRelax bounds. But for INLBRelax, for some classes UT representation taken  minimum CPU time. In Table \ref{table:summaryNLB} we summarize these results.


\begin{table}[H]
\caption{Natural lower bound calculation} \label{NLBtime}
\centering
\resizebox{\columnwidth}{!}{%
\begin{tabular}{@{} l l c c | c c c | c c c | c c c    @{}} \toprule[1.5pt]
\multirow{3}[3]{*}{problem}  & \multicolumn{2}{@{}c@{}}{size}& Opt & \multicolumn{3}{@{}c@{}}{NLB time} &  \multicolumn{3}{@{}c@{}}{NLBRelax time} &  \multicolumn{3}{@{}c@{}}{INLBRelax time}  \\
 \cmidrule[1.2pt](lr){2-3}   \cmidrule[1.2pt](lr){5-7}   \cmidrule[1.2pt](lr){8-10}   \cmidrule[1.2pt](lr){11-13}  &	 m & n & val & min & max & avg  & min & max & avg  & min & max  & avg \\
 \midrule
 \multicolumn{13}{@{}c@{}}{Q non-negative in  [5,10]}\\ \cmidrule[1.2pt](lr){1-13}																		 qsc-m5n20	&	5	&	20	&	61	&	1063	&	4501	&	2227.00	&	578	&	1595	&	1125.33	&	1140	&	1672	 &	1315.00	\\
qsc-m5n30	&	5	&	30	&	27	&	1063	&	4095	&	2841.67	&	797	&	1594	&	1333.50	&	1031	&	1703	 &	1549.67	\\
qsc-m10n20	&	10	&	20	&	115	&	532	&	3657	&	1430.00	&	531	&	1562	&	1023.33	&	938	&	1766	&	 1138.17	\\
qsc-m10n30	&	10	&	30	&	56	&	1782	&	6345	&	3487.50	&	765	&	2126	&	1505.33	&	1453	&	2237	 &	1784.33	\\
qsc-m10n40	&	10	&	40	&	60	&	2500	&	7438	&	5293.17	&	1016	&	2532	&	1956.17	&	2015	&	 2734	&	2292.33	\\
\end{tabular}
 }
\end{table}

\begin{table}[H]
\caption{Natural lower bound calculation} \label{NLBtime1.1.1}
\centering
\resizebox{\columnwidth}{!}{%
\begin{tabular}{@{} l l c c | c c c | c c c | c c c    @{}} \toprule[1.5pt]
\multirow{3}[3]{*}{problem}  & \multicolumn{2}{@{}c@{}}{size}& Opt & \multicolumn{3}{@{}c@{}}{NLB time} &  \multicolumn{3}{@{}c@{}}{NLBRelax time} &  \multicolumn{3}{@{}c@{}}{INLBRelax time}  \\
 \cmidrule[1.2pt](lr){2-3}   \cmidrule[1.2pt](lr){5-7}   \cmidrule[1.2pt](lr){8-10}   \cmidrule[1.2pt](lr){11-13}  &	 m & n & val & min & max & avg  & min & max & avg  & min & max  & avg \\
 \midrule

 \multicolumn{13}{@{}c@{}}{Q positive semidefinite B in  [-5,5]}\\ \cmidrule[1.2pt](lr){1-13}
qsc-m5n20	&	5	&	20	&	4	&	547	&	1187	&	1031.33	&	500	&	1003	&	883.33	&	938	&	1031	&	 987.00	\\
qsc-m5n30	&	5	&	30	&	3	&	797	&	1735	&	1487.33	&	687	&	1468	&	1242.00	&	1297	&	1391	&	 1346.50	\\
qsc-m10n20	&	10	&	20	&	5	&	515	&	1047	&	945.17	&	438	&	891	&	812.83	&	875	&	891	&	880.33	\\
qsc-m10n30	&	10	&	30	&	3	&	750	&	1578	&	1388.17	&	656	&	1344	&	1211.00	&	1313	&	1344	&	 1328.17	\\
qsc-m10n40	&	10	&	40	&	3	&	984	&	2031	&	1828.17	&	938	&	1922	&	1734.67	&	1797	&	1922	&	 1893.67	\\
\end{tabular}
 }
\end{table}

\begin{table}[H]
\caption{Natural lower bound calculation} \label{NLBtime1.1.1}
\centering
\resizebox{\columnwidth}{!}{%
\begin{tabular}{@{} l l c c | c c c | c c c | c c c    @{}} \toprule[1.5pt]
\multirow{3}[3]{*}{problem}  & \multicolumn{2}{@{}c@{}}{size}& Opt & \multicolumn{3}{@{}c@{}}{NLB time} &  \multicolumn{3}{@{}c@{}}{NLBRelax time} &  \multicolumn{3}{@{}c@{}}{INLBRelax time}  \\
 \cmidrule[1.2pt](lr){2-3}   \cmidrule[1.2pt](lr){5-7}   \cmidrule[1.2pt](lr){8-10}   \cmidrule[1.2pt](lr){11-13}  &	 m & n & val & min & max & avg  & min & max & avg  & min & max  & avg \\
 \midrule

 \multicolumn{13}{@{}c@{}}{Q non-negative and positive semidefinite B in  [5,15]}\\ \cmidrule[1.2pt](lr){1-13}
qsc-m5n20	&	5	&	20	&	487	&	1032	&	2219	&	1792.50	&	547	&	1250	&	1045.50	&	1063	&	1375	 &	1258.50	\\
qsc-m5n30	&	5	&	30	&	171	&	1829	&	4395	&	3363.00	&	797	&	1594	&	1429.83	&	1515	&	1703	 &	1648.50	\\
qsc-m10n20	&	10	&	20	&	788	&	547	&	1094	&	987.17	&	500	&	1032	&	924.50	&	922	&	1047	&	 1000.17	\\
qsc-m10n30	&	10	&	30	&	364	&	1578	&	3251	&	2714.00	&	781	&	1594	&	1398.50	&	1453	&	1703	 &	1640.67	\\
qsc-m10n40	&	10	&	40	&	532	&	2641	&	6314	&	5149.00	&	1047	&	2110	&	1890.83	&	2047	&	 2250	&	2190.17	\\
\end{tabular}
 }
\end{table}

\begin{table}[H]
\caption{Natural lower bound  calculation } \label{NLBtime1.3}
\centering
\resizebox{\columnwidth}{!}{%
\begin{tabular}{@{} l l c c | c c c | c c c | c c c    @{}} \toprule[1.5pt]
\multirow{3}[3]{*}{problem}  & \multicolumn{2}{@{}c@{}}{size}& Opt & \multicolumn{3}{@{}c@{}}{NLB time} &  \multicolumn{3}{@{}c@{}}{NLBRelax time} &  \multicolumn{3}{@{}c@{}}{INLBRelax time}  \\
 \cmidrule[1.2pt](lr){2-3}   \cmidrule[1.2pt](lr){5-7}   \cmidrule[1.2pt](lr){8-10}   \cmidrule[1.2pt](lr){11-13}  &	 m & n & val & min & max & avg  & min & max & avg  & min & max  & avg \\
 \midrule
\multicolumn{13}{@{}c@{}}{Q ARB SYM in [-5,5]}\\ \cmidrule[1.2pt](lr){1-13}
qsc-m5n20	&	5	&	20	&	-237	&	547	&	1078	&	976.67	&	468	&	953	&	861.83	&	890	&	938	&	914.33	 \\
qsc-m5n30	&	5	&	30	&	-606	&	812	&	1656	&	1471.33	&	657	&	1375	&	1229.83	&	1297	&	1344	 &	1325.33	\\
qsc-m10n20	&	10	&	20	&	-184	&	516	&	1156	&	963.83	&	453	&	907	&	828.33	&	875	&	938	&	911.67	 \\
qsc-m10n30	&	10	&	30	&	-567	&	750	&	1547	&	1396.17	&	672	&	1359	&	1226.50	&	1313	&	1344	 &	1326.00	\\
qsc-m10n40	&	10	&	40	&	-796	&	1016	&	2032	&	1839.17	&	859	&	1781	&	1612.00	&	1688	&	 1767	&	1750.33	\\
\end{tabular}
 }
\end{table}

\begin{table}[H]
\caption{Natural lower bound  calculation } \label{NLBtime1.2}
\centering
\resizebox{\columnwidth}{!}{%
\begin{tabular}{@{} l l c c | c c c | c c c | c c c    @{}} \toprule[1.5pt]
\multirow{3}[3]{*}{problem}  & \multicolumn{2}{@{}c@{}}{size}& Opt & \multicolumn{3}{@{}c@{}}{NLB time} &  \multicolumn{3}{@{}c@{}}{NLBRelax time} &  \multicolumn{3}{@{}c@{}}{INLBRelax time}  \\
 \cmidrule[1.2pt](lr){2-3}   \cmidrule[1.2pt](lr){5-7}   \cmidrule[1.2pt](lr){8-10}   \cmidrule[1.2pt](lr){11-13}  &	 m & n & val & min & max & avg  & min & max & avg  & min & max  & avg \\
 \midrule

 \multicolumn{13}{@{}c@{}}{Q ARB LS in [-10,5]}\\ \cmidrule[1.2pt](lr){1-13}
qsc-m5n20	&	5	&	20	&	-1160	&	563	&	1188	&	1041.83	&	484	&	1000	&	883.00	&	953	&	1000	&	 976.50	\\
qsc-m5n30	&	5	&	30	&	-2958	&	828	&	1704	&	1510.67	&	703	&	1468	&	1281.17	&	1328	&	1532	 &	1401.17	\\
qsc-m10n20	&	10	&	20	&	-1180	&	532	&	1048	&	940.50	&	453	&	922	&	830.83	&	875	&	923	&	896.00	 \\
qsc-m10n30	&	10	&	30	&	-2651	&	719	&	1547	&	1383.00	&	656	&	1329	&	1213.67	&	1312	&	1328	 &	1315.17	\\
qsc-m10n40	&	10	&	40	&	-4675	&	1032	&	2047	&	1854.33	&	859	&	1766	&	1596.50	&	1734	&	 1766	&	1750.17	\\
\end{tabular}
 }
\end{table}

\begin{table}[H]
\caption{Natural lower bound  calculation } \label{NLBtime1.1}
\centering
\resizebox{\columnwidth}{!}{%
\begin{tabular}{@{} l l c c | c c c | c c c | c c c    @{}} \toprule[1.5pt]
\multirow{3}[3]{*}{problem}  & \multicolumn{2}{@{}c@{}}{size}& Opt & \multicolumn{3}{@{}c@{}}{NLB time} &  \multicolumn{3}{@{}c@{}}{NLBRelax time} &  \multicolumn{3}{@{}c@{}}{INLBRelax time}  \\
 \cmidrule[1.2pt](lr){2-3}   \cmidrule[1.2pt](lr){5-7}   \cmidrule[1.2pt](lr){8-10}   \cmidrule[1.2pt](lr){11-13}  &	 m & n & val & min & max & avg  & min & max & avg  & min & max  & avg \\
 \midrule

 \multicolumn{13}{@{}c@{}}{Q ARB RS in [-5,10]}\\ \cmidrule[1.2pt](lr){1-13}
qsc-m5n20	&	5	&	20	&	-20	&	547	&	1265	&	1021.00	&	484	&	1125	&	903.50	&	937	&	1033	&	 977.00	\\
qsc-m5n30	&	5	&	30	&	-34	&	797	&	1657	&	1454.00	&	688	&	1531	&	1263.50	&	1250	&	1407	&	 1331.17	\\
qsc-m10n20	&	10	&	20	&	13	&	484	&	1031	&	927.00	&	438	&	875	&	781.33	&	844	&	875	&	862.17	\\
qsc-m10n30	&	10	&	30	&	-24	&	781	&	1594	&	1380.17	&	657	&	1281	&	1159.00	&	1250	&	1328	&	 1276.33	\\
qsc-m10n40	&	10	&	40	&	-22	&	969	&	2000	&	1810.00	&	875	&	1751	&	1588.83	&	1703	&	1797	&	 1742.17	\\
\end{tabular}
 }
\end{table}


\begin{table}[H]
\caption{NLB bound calculation} \label{NLBtime2.1}
\centering
\resizebox{\columnwidth}{!}{%
\begin{tabular}{@{} l l c c | c c c | c c c | c c c    @{}} \toprule[1.5pt]
\multirow{3}[3]{*}{problem}  & \multicolumn{2}{@{}c@{}}{size}& Opt & \multicolumn{3}{@{}c@{}}{NLB time} &  \multicolumn{3}{@{}c@{}}{NLBRelax time} &  \multicolumn{3}{@{}c@{}}{INLBRelax time}  \\
 \cmidrule[1.2pt](lr){2-3}   \cmidrule[1.2pt](lr){5-7}   \cmidrule[1.2pt](lr){8-10}   \cmidrule[1.2pt](lr){11-13}  &	 m & n & val & min & max & avg  & min & max & avg  & min & max  & avg \\
 \midrule
 \multicolumn{13}{@{}c@{}}{Rank 1, a in [-10,10] in b in [-5,5]}\\ \cmidrule[1.2pt](lr){1-13}	
qsc-m5n20	&	5	&	20	&	-753	&	1047	&	1140	&	1104.00	&	906	&	1000	&	950.67	&	891	&	1000	 &	956.00	\\
qsc-m5n30	&	5	&	30	&	-1439	&	1548	&	1860	&	1708.67	&	1360	&	1501	&	1435.33	&	1359	&	 1718	&	1465.83	\\
qsc-m10n20	&	10	&	20	&	-454	&	1016	&	1078	&	1047.17	&	937	&	985	&	958.00	&	922	&	985	&	 956.00	\\
qsc-m10n30	&	10	&	30	&	-1061	&	1516	&	1594	&	1567.83	&	1328	&	1375	&	1343.67	&	1344	&	 1422	&	1370.33	\\
qsc-m10n40	&	10	&	40	&	-2291	&	2001	&	2063	&	2023.67	&	1750	&	1766	&	1755.50	&	1735	&	 1798	&	1766.33	\\
\end{tabular}
 }
\end{table}

\begin{table}[H]
\caption{NLB bound calculation} \label{NLBtime2}
\centering
\resizebox{\columnwidth}{!}{%
\begin{tabular}{@{} l l c c | c c c | c c c | c c c    @{}} \toprule[1.5pt]
\multirow{3}[3]{*}{problem}  & \multicolumn{2}{@{}c@{}}{size}& Opt & \multicolumn{3}{@{}c@{}}{NLB time} &  \multicolumn{3}{@{}c@{}}{NLBRelax time} &  \multicolumn{3}{@{}c@{}}{INLBRelax time}  \\
 \cmidrule[1.2pt](lr){2-3}   \cmidrule[1.2pt](lr){5-7}   \cmidrule[1.2pt](lr){8-10}   \cmidrule[1.2pt](lr){11-13}  &	 m & n & val & min & max & avg  & min & max & avg  & min & max  & avg \\
\midrule
 \multicolumn{13}{@{}c@{}}{Rank 2, a in [-10,10] in b in [-5,5]}\\ \cmidrule[1.2pt](lr){1-13}
qsc-m5n20	&	5	&	20	&	-690	&	1172	&	1313	&	1234.33	&	984	&	1125	&	1036.50	&	985	&	1125	 &	1065.33	\\
qsc-m5n30	&	5	&	30	&	-3345	&	1516	&	1797	&	1690.50	&	1297	&	1532	&	1414.00	&	1329	&	 1515	&	1414.17	\\
qsc-m10n20	&	10	&	20	&	-1007	&	984	&	1047	&	1015.50	&	875	&	906	&	887.83	&	875	&	907	&	896.33	 \\
qsc-m10n30	&	10	&	30	&	-786	&	1454	&	1531	&	1497.50	&	1297	&	1313	&	1302.17	&	1281	&	 1313	&	1304.83	\\
qsc-m10n40	&	10	&	40	&	-2725	&	1971	&	2052	&	1999.33	&	1719	&	1814	&	1750.67	&	1719	&	 1798	&	1747.83	\\
\end{tabular}
 }
\end{table}

\begin{table}[H]
\small
\caption{Summary of results for NLB bounds for minimum time (representation)}
\centering
\begin{tabular}{@{} l l l l  @{}} \toprule[1.5pt]
\multirow{3}[3]{*}{Instance class}  &  NLB  &  NLBRelax  &  INLBRelax   \\
	 & representation & representation   & representation  \\
& & &  \\ \midrule

$\matr{Q}$ non-negative  &  UT &  SYM &  UT   \\
$\matr{Q}$ positive semidefinite  &   SYM  & SYM  & UT  \\
$\matr{Q}$ non-negative and positive semidefinite & SYM &  SYM & UT  \\

$\matr{Q}$ arbitrary with symmetric distribution  &    SYM    & SYM    &  -\\
$\matr{Q}$ arbitrary with left-skewed distribution  &  SYM     & SYM   & SYM  \\
$\matr{Q}$ arbitrary with right-skewed distribution  &    SYM     & SYM   & - \\
$\matr{Q}$ Rank 1  &   SYM & SYM  & SYM \\
$\matr{Q}$ Rank 2  &   SYM & SYM  & SYM  \\
\bottomrule[1.5pt]
\end{tabular}
\label{table:summaryNLB}
\end{table}



\section{Conclusion}

In this paper we analyzed commonly used equivalent representations of a QCOP in the context of the quadratic set covering problem. Our experimental analysis using the binary quadratic programming problem solvers of CPLEX and GUROBI demonstrated that CPLEX solver works better under symmetrization while GUROBI performs better with triangularization. Although there are outliers, these conclusions are statistically significant as confirmed by Wilcoxin test. CPLEX was found to be more robust and less sensitive with respect to the equivalent representations considered while GUROBI is somewhat sensitive to these representations and was not able to solve some of the instances that CPLEX solved. However, for instances that GUROBI solved, it terminated faster. Both solvers produced good quality solutions in experimental runs with short time limits (15 minutes and 30 minutes) establishing the value of these solvers as heuristics. \\

We also studied the effect of equivalent representations on the natural lower bound values. It is concluded that the symmetric representation produced stronger lower bounds on the average. This representation also have better running times. For one class of problems, the UT representation produced tighter bounds. \\

Further, we presented various procedures to obtain more involved equivalent representations. In particular, exploiting the seminal works of Billionnet etl~\cite{be1, ba10}, Galli  and Letchford~\cite{gal}, Hammer and  Rubin~\cite{1970Hammer}, and P\"{o}rn et al~\cite{ba2} and linking recent research works on diagonalization (linearization) associated with QCOP and CVP of LCOP, we obtained strong equivalent representations.\\

Our experimental results offer,  guidelines for selecting data representations for standard solvers and we hope to stimulate research on additional experimental analysis, particularly using more sophisticated representations and/or on different family of test problems. We have done some additional experiments using QUBO instead of QSCP and obtained conclusions similar to that are reported here for  the case of QSCP. Finally, our study also provide some benchmark instances for the quadratic set covering problems, which will be made available to interested researchers.\\

\noindent{\textbf{\large Acknowledgement:} This work was supported by an NSERC discovery grant and an NSERC discovery accelerator supplement awarded to Abraham P. Punnen.\\

\appendix

\section{Analysis of generated test instances}
Since the test problems of the type qsc-m\#n\#
In Table \ref{instAna} we provide the analysis of test instances we generated. In the column ``row sum" is  the row sum of the coefficient matrix $\matr{D}$ indicated in section \ref{qspdef} and ``col sum"  to the column sum. Under the column ``rowsum, ``minrow", ``maxrow", and ``avgrow" represent respectively the minimum,  maximum, and average of row sums. Similarly,  under the column ``colsum", ``mincol", ``maxcol", and ``avgcol" represent  respectively the minimum,  maximum, and average of column sums. The value of ``mincol" as zero indicates at least one empty subset and the value of ``minrow" as zero indicates infeasible instance. All instances we generated are feasible. The column ``no.of empty subsets " represents the number of empty subsets for the given instance.

\begin{table}[H]
\caption{Test instance analysis}
\centering
\small
\begin{tabular}{@{} l l c c c c c c c c  @{}} \toprule[1.5pt]
\multirow{3}[3]{*}{problem}  & \multicolumn{2}{@{}c@{}}{size}  & \multicolumn{3}{@{}c@{}}{row sum} & \multicolumn{3}{@{}c@{}}{col sum} &   \multirow{2}[2]{*}{no.of empty subsets} \\
 \cmidrule[1.2pt](lr){2-3}  \cmidrule[1.2pt](lr){4-6} \cmidrule[1.2pt](lr){7-9} &	m & n & minrow  &  maxrow & avgrow & mincol & maxcol &  avgcol  &  \\
 \midrule


qsc-m5n20 & 5 &  20   &  1   &  9   &  5.4   &  0   &  3   &  1.35 &  4\\


qsc-m5n30 & 5 &  30   &  4   &  13   &  8.2   &  0   &  4   &  $\mynum{1.36667}$ &  2\\

qsc-m10n20 & 10 &  20   &  1   &  10   &  5.5   &  0   &  5   &  2.75 &  1\\


qsc-m10n30 & 10 &  30   &  2   &  15   &  10.3   &  0   &  7   &  $\mynum{3.43333}$ & 1\\


qsc-m10n40 & 10 &  40   &  4   &  19   &  10.8   &  0   &  6   &  2.7 &  1\\


qsc-m10n50 & 10 &  50   &  1   &  24   &  11.7   &  0   &  5   &  2.34 &   4\\

qsc-m10n100 & 10 &  100   &  1   &  20   &  12.6   &  0   &  4   &  1.26 &   31\\

qsc-m15n20 & 15 &  20   &  3   &  10   &  6.6   &  3   &  8   &  4.95 &  0\\


qsc-m15n30 & 15 &  30   &  2   &  14   &  $\mynum{7.06667}$   &  1   &  8   &  $\mynum{3.53333}$ &   0\\


qsc-m15n40 & 15 &  40   &  1   &  18   &  $\mynum{10.7333}$   &  1   &  8   &  $\mynum{4.025}$ &  0\\


qsc-m15n50 & 15 &  50   &  1   &  23   &  $\mynum{13.6667}$   &  1   &  8   &  4.1 &  0\\

qsc-m20n20 & 20 &  20   &  1   &  10   &  6.65   &  4   &  9   &  6.65 &  0\\

qsc-m20n30 & 20 &  30   &  1   &  15   &  8.05   &  1   &  8   &  $\mynum{5.36667}$ &  0\\

qsc-m20n40 & 20 &  40   &  2   &  20   &  10.65   &  1   &  10   &  5.325 & 0 \\

qsc-m20n50 & 20 &  50   &  1   &  25   &  14.9   &  2   &  12   &  5.96 &   0\\


qsc-m25n20 & 25 &  20   &  1   &  10   &  5.48   &  3   &  11   &  6.85 & 0 \\

qsc-m25n40 & 25 &  40   &  1   &  20   &  10.72   &  2   &  10   &  6.7 &  0\\

qsc-m25n50 & 25 &  50   &  1   &  24   &  13.84   &  3   &  14   &  6.92 &  0\\

qsc-m30n20 & 30 &  20   &  1   &  10   &  $\mynum{6.03333}$   &  6   &  13   &  9.05 &  0\\

qsc-m30n40 & 30 &  40   &  1   &  20   &  $\mynum{11.7333}$   &  4   &  13   &  8.8 &  0\\

qsc-m30n50 & 30 &  50   &  1   &  25   &  12.4   &  4   &  12   &  7.44 &  0\\

qsc-m30n100 & 30 &  100   &  3   &  50   &  20.6   &  2   &  11   &  6.18 &  0\\

qsc-m40n100 & 40 &  100   &  1   &  49   &  $\mynum{24.075}$   &  4   &  15   &  9.63 &  0\\

qsc-m40n150 & 40 &  150   &  2   &  74   &  $\mynum{34.025}$   &  4   &  15   &  $\mynum{9.07333}$ &  0\\

qsc-m40n200 & 40 &  200   &  2   &  99   &  $\mynum{53.475}$   &  3   &  16   &  $\mynum{10.695}$ & 0 \\

qsc-m40n250 & 40 &  250   &  3   &  124   &  73.55   &  6   &  18   &  $\mynum{11.768}$ &  0\\

qsc-m50n100 & 50 &  100   &  1   &  50   &  26.08   &  5   &  20   &  13.04 &  0\\

qsc-m50n200 & 50 &  200   &  1   &  100   &  47.26   &  4   &  20   &  $\mynum{11.815}$ &  0\\

qsc-m50n300 & 50 &  300   &  2   &  148   &  79.72   &  5   &  23   &  $\mynum{13.2867}$ &  0\\

qsc-m50n350 & 50 &  350   &  9   &  174   &  86.72   &  5   &  24   &  $\mynum{12.3886}$ &  0\\
\bottomrule[1.5pt]
 \end{tabular}
 \label{instAna}
\end{table}

\clearpage

\end{document}